\documentclass[pdflatex,sn-mathphys-num]{sn-jnl}

\usepackage[utf8]{inputenc}
\usepackage{mathtools}
\usepackage{graphicx}%
\usepackage{multirow}%
\usepackage{amsmath,amssymb,amsfonts}%
\usepackage{amsthm}%
\usepackage{parskip}
\usepackage{imakeidx}
\usepackage{hyperref}
\usepackage{mathrsfs}%
\usepackage[title]{appendix}%
\usepackage{xcolor}%
\usepackage{textcomp}%
\usepackage{manyfoot}%
\usepackage{booktabs}%
\usepackage{algorithm}%
\usepackage{algorithmicx}%
\usepackage{algpseudocode}%
\usepackage{listings}%

\usepackage{xparse, etoolbox}
\DeclarePairedDelimiterX\innerp[2]{\langle}{\rangle}{#1,#2}

\DeclareMathOperator{\supp}{supp}

\theoremstyle{thmstyleone}%
\newtheorem{theorem}{Theorem}[section]
\newtheorem{proposition}[theorem]{Proposition}%
\newtheorem{corollary}[theorem]{Corollary}
\newtheorem{lemma}[theorem]{Lemma}

\theoremstyle{thmstyletwo}%
\newtheorem{remark}[theorem]{Remark}%

\theoremstyle{thmstylethree}%
\newtheorem{definition}{Definition}[section]%

\begin{document}

\title[Article Title]{Nonexistence of finite-time blow-up for the equivariant harmonic map heat flow from \texorpdfstring{$B^2$}{B2} to \texorpdfstring{$S^2$}{S2}}

\author*{\fnm{Dylan} \sur{Samuelian}}\email{dylan.samuelian@epfl.ch}

\affil*{\orgdiv{SB MATH PDE}, \orgname{EPFL}, \orgaddress{\street{Station 8}, \city{Lausanne}, \postcode{1015}, \country{Switzerland}}}


\abstract{We consider $D$-equivariant solutions to the harmonic map heat flow from $B^2$ to $S^2$ under general time-dependent smooth boundary data and prove that there is no finite-time blow-up when $D \geq 2$. In contrast with the corresponding problem on the whole plane, the bounded domain and boundary condition are not preserved by the scaling of the equation, while the time-dependent boundary condition introduces an energy flux through the boundary, so that the total energy is no longer monotone. To overcome these difficulties, a careful localized blow-up analysis combining profile decomposition, modulation estimates and the collision-interval method is required.}

\keywords{harmonic map heat flow, bubble tree, blow-up, parabolic, nonlinear, soliton resolution, equivariant}

\maketitle

\tableofcontents

\section{Introduction}
We are interested in the behaviour of finite-time and finite-energy solutions to the Harmonic Map Heat Flow from the 2D-ball to the 2-sphere
\begin{align}
    \Phi_t &= \Delta \Phi + |\nabla \Phi|^2 \Phi, \quad (x,t) \in B^2 \times  (0,+\infty)  \\
    \Phi(x,t) &= \Phi_0(x,t), \quad (x,t) \in \partial B^2 \times [0,+\infty) \cup B^2 \times \{0\}, \notag
\end{align}
where $\Phi : B^2 \rightarrow S^2 \subset \mathbb R^3$, $\nabla \Phi = (J\Phi)^T$ is the transposed Jacobian matrix, the energy is defined as
\begin{equation}
    E(\Phi(\cdot,t)) = E(\Phi(t)) = \frac{1}{2}\int_{B^2} |\nabla \Phi(\cdot,t)|^2 dx 
\end{equation}
and the boundary data is the restriction of a function $\Phi_0 \in C^{\infty}(\overline{B^2} \times [0,+\infty); S^2)$ which is smooth and $D$-equivariant, meaning it has the form:
\begin{equation}
    \Phi_0(re^{i\theta},t) = \left( e^{iD\theta} \sin v_0(r,t), \cos v_0(r,t) \right) \in \mathbb C \times \mathbb R \simeq \mathbb R^{3}, 
\end{equation}
with $v_0(r,t) = r^D \tilde{v}_0(r,t)$, $\tilde{v}_0(r,t) \in C^{\infty}([0,1] \times [0,+\infty),\mathbb R)$, being the \textit{inclination coordinate} of $\Phi_0$. It follows that $\Phi(x,t)$ is also $D$-equivariant with an inclination coordinate $v(r,t)$ solving the following one-dimensional nonlinear heat equation \begin{align}
v_t &= v_{rr} + \frac{v_r}{r} - D^2 \frac{\sin(2v)}{2r^2}, \quad (r,t) \in (0,1) \times (0,+\infty), \\
v(r,t) &= v_0(r,t), \quad (r,t) \in \{0,1\} \times [0,+\infty) \cup (0,1) \times \{0\}. \notag
\end{align}

Assuming blow-up, Struwe (\cite{struwe2008variational}) described the formation of both finite-time and global-time singularities for $\Phi$ in terms of concentration of energy and he showed that one can extract a harmonic map from $S^2$ to $S^2$ near such singularities by choosing an appropriate sequence of times (see (\ref{eq: well-chosen sequence of times})) along which the flow becomes stationary. The description of the blow-up event was further improved by Qing (\cite{Qing1995OnSO}), who showed that near a singularity, the solution decomposes asymptotically as a sum of rescaled harmonic maps and a radiation term, using the so-called Palais-Smale condition. 

Finite-time blow-up was shown to exist in the $1$-equivariant setting by Chang, Ding and Ye (\cite{Chang1992FinitetimeBO}) and infinite-time blow-up was exhibited by Chang, Ding for $D = 1$ (\cite{ChangDing2011}),  Gustafson, Nakanishi and Tsai for $D  = 2$ (\cite{stability}, in the whole plane setting) and the result from Chang, Ding was generalized for $D \geq 1$ by the author (\cite{samuelian2025blowuptrees}). However, the precise form of the decomposition (How many bubbles ? Is the decomposition unique along all possible time sequences ? Is the decomposition continuous in time ?) remained unknown for a long time. In the $1$-equivariant setting on $B^2$, Van der Hout (\cite{VANDERHOUT2003}) was able to prove that only single-bubbling can occur and the present author generalized his result for $D \geq 1$ (\cite{samuelian2025blowuptrees}). 

In the $\mathbb R^2$ case instead of $B^2$, it has been recently proved by Jendrej and Lawrie (\cite{jendrej}) that the bubbling happens continuously in time (in the $D$-equivariant setting) using ideas inspired by the soliton resolution for dispersive equations. In the general nonsymmetric case, one only has a partial result, due to Jendrej, Lawrie and Schlag (\cite{Jendrej_Lawrie_Schlag_2025}): along any sequence of times, one can find a subsequence along which bubbling happens. Even more recently, it has been shown by Kim and Merle (\cite{kim2025classificationglobaldynamicsenergycritical}) that there is no finite-time blow-up when $D \geq 3$, still in the $\mathbb R^2$ case with $v_0(x,t) = v_0(x)$. This was already hinted by a previous stability result from Gustafson, Nakanishi and Tsai (\cite{stability}) for $D \geq 3$. 

The aim of this paper is to prove nonexistence of finite-time blow-up for the $D$-equivariant, $D \geq 2$, harmonic map heat flow from $B^2$ to $S^2$ (Theorem \ref{thm:main thm}) under smooth time-dependent boundary data. The proof is inspired by the arguments of Kim and Merle \cite{kim2025classificationglobaldynamicsenergycritical}, as well as Jendrej and Lawrie \cite{jendrej}. We note that it is already known from \cite{VANDERHOUT2003} and the author's work \cite{samuelian2025blowuptrees} that there can be only one bubble in such a profile decomposition. Therefore, the blow-up analysis simplifies substantially. However, our problem is still different in two important aspects. First, the whole-plane results cannot be applied directly because the parabolic scaling is not a symmetry of the problem on $B^2$ (the scaling does not preserve the domain and the boundary condition). However, one can still use the scaling isometries of $\dot{H}^1$ after localization of $v(r,t)$ to recover a profile decomposition near the origin and use the local energy arguments from Struwe (\cite{struwe2008variational}) to deduce further properties on the decomposition. Second, the aforementioned works only treat the initial-value problem, where one has monotonicity of the energy. Even though the blow-up phenomenon is local, the energy is non-monotone in our setting and the boundary data cannot be ignored a priori in the analysis of blow-up. Indeed, the comparison principle shows that if the boundary data satisfies $|v_0| \leq \pi$, then a finite-time blow-up cannot occur (\cite{ChangDing2011}, \cite{samuelian2025blowuptrees}). Moreover, there are examples of finite-time blow-up for $D = 1$ (\cite{Chang1992FinitetimeBO}) for the boundary-value problem on $B^2$ and this existence result also relies on a boundary criterion (take $v_0(r,t) = v_0(r), v_0(0) = 0, |v_0(1)| > \pi$). Still, the single-bubble result from \cite{VANDERHOUT2003} and \cite{samuelian2025blowuptrees} combined with the profile decomposition analysis in Section \ref{sec: profile decompo along sequence of times} hints that the boundary data has limited influence on the blow-up dynamics. Hence, the nonexistence result for $D \geq 3$ was also expected for the boundary-value problem. This still left the case $D = 2$ unresolved. However, global existence was hinted: it is already known that solutions to the $m$-equivariant Yang-Mills heat flow in 4D, which corresponds to a $2$-equivariant harmonic map heat flow with a different geometry when $m = 1$ (\cite{grotowski2007geometric}), are global (\cite{hong2004global}). The borderline $D=2$ case is handled in Theorem \ref{thm:modulation estimates} by using a cancellation in $(\Lambda +2)\Lambda W$, which yields an improved decay $(\Lambda +2)\Lambda W \lesssim \langle y \rangle^{-2D-2}$ compared to the class $D \geq 3$, and by using the logarithmic growth of $||r \Lambda W||_{L^2(B_{\lambda^{-1}}(0))}$. 

 

\begin{theorem}[Nonexistence of finite-time blow-up]\label{thm:main thm}
    Consider a solution $v(r,t)$ of the $D$-equivariant harmonic map heat flow
     \begin{align}
v_t &= v_{rr} + \frac{v_r}{r} - D^2 \frac{\sin(2v)}{2r^2}, \quad (r,t) \in (0,1) \times (0,+\infty), \label{eq: radial HMF, 2-dim pb} \\
v(r,t) &= v_0(r,t), \quad (r,t) \in \{0,1\} \times [0,+\infty) \cup (0,1) \times \{0\} \notag
\end{align}
with smooth boundary and initial data $v_0(r,t) = r^D \tilde{v}_0(r,t)$, $\tilde{v}_0 \in C^{\infty}([0,1]\times [0,+\infty))$. If $D \geq 2$, then $v(r,t)$ is global.
\end{theorem}

\begin{remark}
    Such a solution necessarily has finite energy on any finite time-interval $[0,T)$, $T <+\infty$, i.e.,
    $$
    \sup_{t \in [0,T)} E(v(\cdot,t)) := \sup_{t \in [0,T)}  \pi \int_0^1 \left( v_r^2 + \frac{D^2}{r^2} \sin(v)^2 \right) r dr < +\infty.
    $$
   This follows by differentiating the energy with respect to $t$. The energy is not monotone, but boundary terms can be controlled since $v_0$ is smooth.  See \cite[Remark 4.1]{samuelian2025blowuptrees}.
\end{remark}

\begin{remark}
    A blow-up at infinity is still possible.
\end{remark}


The proof goes as follows. First, we need an abstract profile decomposition, similar to that of Bahouri and Gérard (\cite{bahourigerard1999}). Such an abstract profile decomposition theorem has been developed by Fieseler, Schindler and Tintarev (Theorem \ref{thm:abstract profile decomposition}) for Banach and Hilbert spaces on which a set of isometries act by \textit{dislocations}. This decomposition is the first step needed to prove the profile decomposition (or soliton resolution) for $v(r,t)\chi_{1/4}(r)$ along a well-chosen sequence of times using the Palais-Smale condition, as in the work of Struwe (\cite{struwe2008variational}) and Qing (\cite{Qing1995OnSO}). Indeed, such a decomposition was already known for $\Phi$ (\cite{Qing1995OnSO}), but it is not immediately clear that it should imply a profile decomposition for the inclination coordinate $v$. Jendrej and Lawrie (\cite{jendrej}) proved the continuous-in-time profile decomposition for $v(r,t)$ in the whole space setting (\cite{jendrej}), but it is not clear that it should also apply to the case of a bounded domain with a boundary condition. 

Along one such sequence of times, one can then take inspiration from the argument of Kim and Merle (\cite{kim2025classificationglobaldynamicsenergycritical}) to prove that the scaling parameters cannot converge too fast to zero. This is a combination of a coercivity estimate for the Hamiltonian, an orbital stability result, an $L^2$-estimate for the remainder of the orbital stability result and a final modulation estimate on the scale $\lambda(t)$ obtained in the orbital stability result. 

In the last section, we prove that, given the profile decomposition along a sequence of times and given the modulation estimates, the scale from orbital stability can be taken continuously up to the blow-up time. In other words, $v(r,t)-v(r,T)$ stays continuously close to the $1$-parameter family $Q_{\lambda}$. This is the last step needed to reach a contradiction to the existence of finite-time blow-up for $D \geq 2$. For this last step, we follow the collision-interval argument from Jendrej and Lawrie (\cite{jendrej}).

\section{Notation}
We consider $v(r,t)$ solving (\ref{eq: radial HMF, 2-dim pb}) and blowing-up at time $T < +\infty$. In particular, the energy
$$ 
\sup_{t \in [0,T)} E(v(\cdot,t)) := \sup_{t \in [0,T)} \pi \int_0^1 \left( v_r^2 + \frac{D^2}{r^2} \sin(v)^2 \right) r dr  < +\infty
$$
is finite. Note that $v$ corresponds to a radial function $v(y,t)$ with $y \in \overline{B_1^{(2)}(0)} \subset \mathbb R^2$. Letting $v(r,t) = r^Du(r,t)$, the equation for $u(r,t)$ reads as
\begin{equation}
    \partial_t u = \partial_{rr} u + \frac{2D+1}{r} \partial_r u + D^2 r^{-D-2} \left( \frac{2r^Du - \sin(2r^Du)}{2} \right) \label{eq: radial HMF, 2k+2-dim pb}
\end{equation}
and $u(r,t)$ corresponds to a radial function $u(y,t)$ with $y \in \overline{B_1^{(2D+2)}(0)} \subset \mathbb R^{2D+2}$, whose energy is 
$$
E(u(\cdot,t))=\pi\int_0^1\left[\left(\partial_r u+\frac{D}{r}u\right)^2+\frac{D^2}{r^{2D+2}}\sin\!\left(r^D u\right)^2\right]\,r^{2D+1}dr .
$$

If $v(r,t)$ is a solution of (\ref{eq: radial HMF, 2-dim pb}), then the rescaled function $v_{\lambda}(r,t) =  v(r\lambda^{-1},t\lambda^{-2})$, $\lambda > 0$, is also a solution of  (\ref{eq: radial HMF, 2-dim pb}). In terms of $u$, the appropriate family of solutions for (\ref{eq: radial HMF, 2k+2-dim pb}) is obtained via $v_{\lambda} = r^D u_{\lambda}$, i.e.,
\[
u_{\lambda} = r^{-D} v(r \lambda^{-1},t\lambda^{-2}) = \lambda^{-D} (r\lambda^{-1})^{-D} v(r \lambda^{-1},t\lambda^{-2}) = \lambda^{-D} u(r\lambda^{-1},t\lambda^{-2}).
\]
The infinitesimal generators with respect to these transformations are given by:
\begin{align*}
    \Lambda^{(2)} v(r) &= - \left[ \partial_{\lambda} v_{\lambda} \right]_{\lambda = 1} = r \partial_r v(r), \quad v(r) = v(y), y \in \mathbb R^2, \\
     \Lambda u(r) &= - \left[ \partial_{\lambda} u_{\lambda} \right]_{\lambda = 1} = D u(r) + r \partial_r u(r), \quad u(r) = u(y), y \in \mathbb R^{2D+2}.
\end{align*}
We also introduce a few notations : $Q(r) = 2\arctan(r^D)$, $W(r) = r^{-D}Q(r)$, 
\[
f(v)  = D^2 \left( \frac{2v-\sin(2v)}{2} \right), \quad H = - \Delta - r^{-2}f'(Q) = - \Delta - 8D^2 \frac{r^{2D-2}}{(1+r^{2D})^2}.
\]
Here, $H$ is the operator arising in the linearization of (\ref{eq: radial HMF, 2k+2-dim pb}) around the ground state $W$. The kernel of $H$ in $\dot{H}^1(B_1(0))$ and $\dot{H}^1(\mathbb R^{2D+2})$ is given by 
\[
\Gamma_1 = \Lambda W = \frac{2D}{1+r^{2D}}.
\]
Another linearly independent solution is found using the method of variation of constants:
\[
\Gamma_2 = \Lambda W \cdot \int_1^r \frac{1}{s^{2D+1} (\Lambda W)^2} ds, \quad |\Gamma_2| \sim r^{-2D}, \quad |r| \ll 1
\]
and this solution is not in $\dot{H}^1(B_1(0))$.

Let also $\chi_{r \leq a} = \chi(r/a)$, $a > 0$, denote a smooth cut-off function, $|\chi| \leq 1$, which is $1$ on $r \leq a$, $0$ on $r \geq 2a$ and 
\begin{equation}
    \sup_{1\leq k \leq n}\sup_{r \geq 0}a^k |\partial_r^{(k)} \chi_{r \leq a}| \leq B_{\chi}^{(n)}. \label{eq: upper bound cut-off}
\end{equation}

\section{Profile Decomposition along a sequence of times}\label{sec: profile decompo along sequence of times}
In what follows, we will work at the level of a finite-energy solution $v(r,t)$ for (\ref{eq: radial HMF, 2-dim pb}). Assume $T < +\infty$. First, we show that finite-time blow-up leads to an asymptotic profile decomposition.

\subsection{Profile extraction}
As in \cite[Chapter 3, Lemma 5.9]{struwe2008variational}, if 
$$
\Phi(re^{i\theta},t) =  \left( e^{iD\theta} \sin v(r,t), \cos v(r,t) \right) \in \mathbb C \times \mathbb R \simeq \mathbb R^{3},
$$
then for any $0 < t_1 \leq t \leq t_2 < T$ and $\chi \in C^{\infty}_c(B_1(0))$ radial, an integration by parts shows that
    \begin{align}
   \int_{B_1(0) }  |\Phi_t|^2\chi^2 +  \frac{1}{2}\frac{d}{dt} \left( |\nabla \Phi|^2 \chi^2 \right) dx &= -  \int_{B_1(0) } 2 \chi (\nabla \chi)^T \nabla \Phi \Phi_t dx   \label{eq:equality involving v_t} \\
    &\leq \frac{1}{2}\int_{B_1(0) } |\Phi_t|^2\chi^2dx + 2 \int_{B_1(0) }  |\nabla \Phi|^2|\nabla \chi|^2 dx. \notag 
    \end{align}
Further assuming that $\chi = 1$ on some interval $[a,b] \subset [0,1)$, this implies, at the level of $v$, that
    \begin{align*}
    \int_{t_1}^{t_2} \int_{a}^b  \frac{1}{2}|v_t|^2rdrdt &\leq 4 \int_{t_1}^{t_2}  \int_{0}^1  |v_r|^2|\chi_r|^2 rdr dt \notag \\
    &\leq 4(t_2-t_1) \cdot \sup_{t \in [0,T]} E(v;[0,1]) \cdot ||\chi_r||_{L^{\infty}([0,1])}^2.
    \end{align*}
In particular,
\begin{equation}
    \int_0^T\int_{a}^b|v_t|^2rdrdt < +\infty, \label{eq: finite L^2-norm of v_t}
\end{equation}

Thanks to (\ref{eq: finite L^2-norm of v_t}) with $[a,b] = [0,1/2]$, one can choose a sequence of times $t_n \to T$ along which
\begin{equation}
    \sqrt{T-t_n}||v_t(t_n)||_{L^2(B_{1/2}(0))} \to 0. \label{eq: well-chosen sequence of times}
\end{equation}
Up to taking a subsequence of times, as $v(r,t_n)\chi_{1/4}(r)$ is a bounded sequence of radial functions in $\dot{H}^1_{rad}(\mathbb R^2)$, one obtains an abstract profile decomposition 
$$
v(r,t_k)\chi_{1/4}(r) = v(r,T)\chi_{1/4}(r) + \sum_{m=1}^M \psi_{m} \left(\frac{r}{\lambda_{m,k}} \right) + \varepsilon_{M,k}(r), \quad k \geq M, M \in \mathbb N_{\geq 0},  r > 0,
$$
using the scaling dislocations of $\dot{H}^1_{rad}(\mathbb R^2)$ (see Corollary \ref{thm:abstract profile decomposition with weak limit first}). The decomposition satisfies:
\begin{enumerate}
    \item \textbf{Separation of scales:} For $l \neq j$,
    $$
\lim_{k \to +\infty} \frac{\lambda_{l,k}}{\lambda_{j,k}} + \frac{\lambda_{j,k}}{\lambda_{l,k}} = +\infty
    $$
    with the convention that $\lambda_{0,k} = 1$. 
    \item \textbf{Weak convergence to profile:} The weak limit in $\dot{H}^1(\mathbb R^2)$ of $v(\lambda_{m,k} r,t_k)\chi_{1/4}(\lambda_{m,k} r)$ is $\psi_{m}(r)$.
    \item \textbf{Smallness of remainder:}  For any $0 \leq m \leq M$, the weak limit in  $\dot{H}^1(\mathbb R^2)$ of $\varepsilon_{M,k}(\lambda_{m,k} r)$ is zero.
 \item \textbf{Pythagorean decomposition of energy:} For any $M \geq 0$,
\begin{align}
        \limsup_{k \to +\infty} ||v(t_k)\chi_{1/4}(\cdot)||_{\dot{H}^1(\mathbb R^2)}^2  =  || v(T)\chi_{1/4}(\cdot)||_{\dot{H}^1(\mathbb R^2)}^2 &+ \sum_{m=1}^M||\psi_m||_{\dot{H}^1(\mathbb R^2)}^2 \notag \\
        &+ \limsup_{k \to +\infty} ||\varepsilon_{M,k}||_{\dot{H}^1(\mathbb R^2)}^2. \notag \\
         \limsup_{k \to +\infty} ||(v(t_k)-v(T))\chi_{1/4}(\cdot)||_{\dot{H}^1(\mathbb R^2)}^2  &=  \sum_{m=1}^M||\psi_m||_{\dot{H}^1(\mathbb R^2)}^2 \notag \\
         &+ \limsup_{k \to +\infty} ||\varepsilon_{M,k}||_{\dot{H}^1(\mathbb R^2)}^2. \label{eq:pythagorean decompo energy profile}
\end{align}
    \item \textbf{Finiteness of decomposition:} Let
    \begin{equation}
    D[(\varepsilon_n)] := \sup\{||\psi||_{\dot{H}^1(\mathbb R^2)} : \exists \psi \in \dot{H}^1(\mathbb R^2), \exists (n_k)_{k \in \mathbb N} \subset \mathbb N, (\lambda_k)_{k \in \mathbb N} \subset \mathbb R_{>0}, \varepsilon_{n_k}(\lambda_k r) \to \psi \text{ weakly } \}. \label{eq: finiteness criterion decompo}
\end{equation}
    For $M \in \mathbb N_{\geq 0}$, the extracted profile $\psi_{M+1}$ is non-zero if and only if $D[(\varepsilon_{M,n})] > 0$. Moreover, if $D[(\varepsilon_{M,n})] = 0$, then all the subsequent profiles $\psi_{M+1} = \psi_{M+2} = ... = 0$ are zero and the decomposition is finite. 
\end{enumerate}

\begin{remark}
    One actually has $\lambda_{M,k} \to 0$ if $M \neq 0$ and $\psi_M$ is a non-zero profile. Indeed, recall that $\psi_M$ is the weak $\dot{H}^1_{rad}(\mathbb R^2)$ limit of the rescaled sequence
\[
v_k(r) :=   v (\lambda_{M,k} r,t_k) \chi_{1/2}(\lambda_{M,k} r).
\]
The support of this function is $\lambda_{M,k}^{-1} B_1(0)$. If $\lim \limits_{k \to +\infty} \lambda_{M,k} = +\infty$, then the support shrinks to a zero-measure set, while the profile is non-zero in $\dot{H}^1(\mathbb R^2)$. This is a contradiction because
$$
|\langle v_k, f \rangle|_{\dot{H}^1(\mathbb R^2)} \leq \sup_k ||v_k||_{\dot{H}^1(\mathbb R^2)} \cdot ||\nabla f||_{L^2(K \cap \supp(f))} \to 0 
$$
for any test function $f \in C^{\infty}_c(\mathbb R^2)$, which would imply a zero profile.
\end{remark}

\begin{remark}\label{rmk: pythagorean decomposition on any ball}
    The Pythagorean decomposition of energy holds on any open subset $U \subset \mathbb R^2$ containing the origin. Indeed, it suffices to show that the cross-terms
    $$
    \langle \psi_l(\lambda_{l,k}^{-1} \cdot) , \psi_j (\lambda_{j,k}^{-1} \cdot) \rangle_{\dot{H}^1(U)}, \ l \neq j, \quad   \langle \psi_l(\lambda_{l,k}^{-1} \cdot), \varepsilon_{M,k}(\cdot) \rangle_{\dot{H}^1(U)}
    $$
    converge to zero as $k \to +\infty$, using the convention that $\lambda_{0,k} = 1, \psi_0 = v(r,T)\chi_{1/4}(r)$.  Observe first that
    \begin{align*}
          |\langle \psi_l(\lambda_{l,k}^{-1} \cdot) , \psi_j (\lambda_{j,k}^{-1} \cdot) \rangle_{\dot{H}^1(U)}| &\leq |\langle f_l(\lambda_{l,k}^{-1} \cdot) , \psi_j (\lambda_{j,k}^{-1} \cdot) \rangle_{\dot{H}^1(U)}| +  |\langle (f_l-\psi_l)(\lambda_{l,k}^{-1} \cdot) , \psi_j (\lambda_{j,k}^{-1} \cdot) \rangle_{\dot{H}^1(U)}|  \\
  &\leq |\langle f_l(\lambda_{l,k}^{-1} \cdot) , \psi_j (\lambda_{j,k}^{-1} \cdot) \rangle_{\dot{H}^1(U)}| +  ||f_l-\psi_l||_{\dot{H}^1(\mathbb R^2)} \cdot ||\psi_j||_{\dot{H}^1(\mathbb R^2)} \\
  &\leq |\langle f_l(\lambda_{l,k}^{-1} \cdot) , f_j (\lambda_{j,k}^{-1} \cdot) \rangle_{\dot{H}^1(U)}| +  ||f_l-\psi_l||_{\dot{H}^1(\mathbb R^2)} \cdot ||\psi_j||_{\dot{H}^1(\mathbb R^2)}  \\
  &+ ||f_j-\psi_j||_{\dot{H}^1(\mathbb R^2)} \cdot ||f_l||_{\dot{H}^1(\mathbb R^2)} 
    \end{align*}
    by Cauchy-Schwarz and scaling invariance in $\mathbb R^2$. As $\dot{H}^1(\mathbb R^2)$ is the completion of $C^{\infty}_c(\mathbb R^2)$ in the homogeneous Sobolev norm, one can approximate $\psi_{l}$ (resp. $\psi_j$) by a smooth, compactly supported function $f_{l}$ (resp. $f_{j}$) for which
    $$
||f_{l} - \psi_{l}||_{\dot{H}^1(\mathbb R^2)} \leq \frac{\varepsilon}{6(1 + \max \limits_{m = 1,...,M}||\psi_{m}||_{\dot{H}^1(\mathbb R^2)})}, \quad ||f_l||_{\dot{H}^1(\mathbb R^2)} \leq 2 ||\psi_l||_{\dot{H}^1(\mathbb R^2)}.
    $$
Then, it suffices to show that 
$$
\limsup_{k \to +\infty}  |\langle f_l(\lambda_{l,k}^{-1} \cdot) , f_j (\lambda_{j,k}^{-1} \cdot) \rangle_{\dot{H}^1(U)}|  \leq \varepsilon/3
$$
and let $\varepsilon \to 0$. For this, write
    $$
     \langle f_l(\lambda_{l,k}^{-1} \cdot) , f_j (\lambda_{j,k}^{-1} \cdot) \rangle_{\dot{H}^1(U)} = \int_{\lambda_{j,k}^{-1} U} \frac{\lambda_{j,k}}{\lambda_{l,k}} \nabla f_l \left( \frac{\lambda_{j,k}}{\lambda_{l,k}}  x \right) \nabla f_j(x) dx, \quad  \lim_{k \to +\infty}\frac{\lambda_{l,k}}{\lambda_{j,k}}  = +\infty,
    $$
   and use dominated convergence. As for the remainder, we also approximate the profile by a smooth, compactly supported function. Then
    $$
    \langle \varepsilon_{M,k},f_j(\lambda_{j,k}^{-1} \cdot)  \rangle_{\dot{H}^1(U)} = \int_{\lambda_{j,k}^{-1} U} \lambda_{j,k} \nabla \varepsilon_{M,k} \left( \lambda_{j,k}  x \right) \nabla f_j(x) dx.
    $$
    If $j \neq 0$, then $\lambda_{j,k}^{-1}(U) \cap \supp(f) = \supp(f)$ for all $k$ large enough as $U$ contains the origin. If $j = 0$, the support $U$ is fixed. Let $D \in \{U, \supp(f)\}$ denote this fixed domain of integration. Then,
\begin{align*}
        \langle \varepsilon_{M,k},f_j(\lambda_{j,k}^{-1} \cdot)  \rangle_{\dot{H}^1(U)} &= \int_{D} \lambda_{j,k} \nabla \varepsilon_{M,k} \left( \lambda_{j,k}  x \right) \nabla f_j(x) dx \\
        &= \int_{\mathbb R^2} \lambda_{j,k} \nabla \varepsilon_{M,k} \left( \lambda_{j,k}  x \right) 1_{D}(x) \nabla f_j(x) dx.
\end{align*}
     As $\lambda_{j,k} \nabla \varepsilon_{M,k} \left( \lambda_{j,k}  x \right) \to 0$ weakly in $L^2(\mathbb R^2;\mathbb R^2)$ and $1_{D}(x) \nabla f_j(x) \in L^2(\mathbb R^2;\mathbb R^2)$, we are finished.
\end{remark}

\subsection{Identifying the profiles and the rates of concentration}
In this section, we show that the profiles $\psi_m$ are non-constant harmonic maps and the concentration scales satisfy $\lambda_{m,k} \ll \sqrt{T-t_k}$. Let
\begin{align*}
    E(v(\cdot,t),[0,1/4]) &= \pi \int_0^{1/4} \left( v_r^2 + \frac{D^2}{r^2} \sin(v)^2 \right) r dr, \\
    E'(v(\cdot,t),[0,1/4])[\eta] &= -2\pi \int_0^{1/4} v_t(t,r)\eta(r)rdr, \eta \in C^{\infty}_c(B_{1/4}(0)).
\end{align*}
where we recall that the energy satisfies a triangle inequality $$
E(v + w)^{1/2} \leq E(v)^{1/2} + E(w)^{1/2}
$$
which implies the reverse triangle inequality 
$$
|E(v)^{1/2}-E(w)^{1/2}| \leq E(v-w)^{1/2}
$$
as well.

Consider a non-zero profile $\psi_m$, $m > 0$, and the associated scale $\lambda_{m,k} \to 0$. We know that $v(\lambda_{m,k} r, t_k) \to \psi_m(r)$ "weakly" in $\dot{H}^1(\mathbb R^2)$.

\begin{lemma} \label{lemma: lambda_k lessim sqrt T-t_k}
    One has $\lambda_{m,k} \lesssim \sqrt{T-t_k}$.
\end{lemma}

\begin{proof}
Assume not. Fix $0 < a < b < +\infty$ on which $\psi_m$ is non-constant. Let $V_k(s,\rho) = v(\lambda_{m,k} \rho,T + \lambda_{m,k}^2 s)$ solve (\ref{eq: radial HMF, 2-dim pb}) on $[-1,0] \times [a/2,2b]$ for $k$ large enough. As the nonlinearity of (\ref{eq: radial HMF, 2-dim pb}) is bounded on that region, parabolic regularity (e.g. see the Parabolic Sobolev Embedding in \cite{samuelian2025blowuptrees}) shows that $V_k$ is Hölder-continuous on $[-1,0] \times [a,b]$, uniformly in $k$. Let 
$$
s_k = -(T-t_k)/\lambda_{m,k}^2 \to 0.
$$
Then,
$$
v(t_k,\lambda_{m,k} \rho) - v(T,\lambda_{m,k}\rho) = V_k(s_k,\rho) - V_k(0,\rho) \to 0, \quad v(T,\lambda_{m,k}\rho) \to \pi \mathbb Z
$$
uniformly in $\rho \in [a,b]$,  which contradicts the fact that the weak $\dot{H}^1(B_b(0) \setminus B_a(0))$-limit $\psi_m$ of $v(t_k,\lambda_{m,k} \rho)$ is non-constant.
\end{proof}

\begin{proposition}[Finiteness of the profile decomposition]
    There is a finite number of non-constant profiles. Moreover, each profile is of the form:
    \begin{equation} \label{all possible form for limiting H}
    \tilde{m}\pi \pm 2\arctan \left( (\alpha r)^D \right), \alpha > 0.
\end{equation}
\end{proposition}
\begin{proof}

Take any $\eta(r) \in C^{\infty}_c((0,+\infty))$, fix $m$ and let $\eta_k(r) = \eta(r/\lambda_{m,k})$. For $k$ large enough, the support of $\eta_k$ is a subset of $(0,1/4)$. By Cauchy-Schwarz, we find that
\begin{align}
    |E'(v(t_k))[\eta_k]| &\lesssim ||v_t(t_k)||_{L^2([0,1/4],rdr)} \cdot ||\eta_k||_{L^2([0,1/4],rdr)} \notag \\
    &\lesssim \lambda_{m,k} ||v_t(t_k)||_{L^2([0,1/4],rdr)} \cdot ||\eta||_{L^2((0,+\infty),rdr)} \label{eq:palais smale condition}\\
    &\lesssim \sqrt{T-t_k} ||v_t(t_k)||_{L^2([0,1/4],rdr)} \cdot ||\eta||_{L^2((0,+\infty),rdr)} \to 0. \notag
\end{align}
This is the so-called Palais-Smale condition. Let $v_k(r) = v(\lambda_{m,k}r,t_k)$. This rewrites as
$$
DE(v_k)[\eta] := 2\pi \int_0^{\infty} \left( v_k'(r)\eta'(r) + \frac{D^2}{r^2}\sin(v_k(r))\cos(v_k(r))\eta(r) \right)rdr \to 0.
$$
Moreover,
$$
2\pi \int_0^{\infty} v_k'(r)\eta'(r) rdr \to 2\pi \int_0^{\infty} \psi_m'(r) \eta'(r) rdr
$$
by weak-convergence. As $v$ is bounded in $L^{\infty}$ (by comparison principle, see \cite{samuelian2025blowuptrees}, as the boundary and initial data $v_0$ is bounded by a large multiple of $\pi$ when $T < +\infty$), $v_k$ is a bounded sequence in $H^1_{loc}((0,+\infty)) \hookrightarrow C^{0,1/4}_{loc}((0,+\infty)) \subset \subset C^0_{loc}((0,+\infty))$. Thus, it converges weakly in $H^1_{loc}((0,+\infty))$ up to taking a further subsequence. Hence, the convergence is uniform on any compact set and the weak $H^1(\mathbb R^2)$ limit of $v_k$ is $\psi_m$ (up to changing $\psi_m$ by $\psi_m + c_m$ for some appropriate constant).

As $s \mapsto \sin s \cos s$ is Lipschitz on a bounded set containing the image of $v(t_k)$, we deduce that 
$$
DE(\psi_m)[\eta] = 2\pi \int_0^{\infty} \left(  \psi_m'(r)\eta'(r) + \frac{D^2}{r^2}\sin(\psi_m(r))\cos(\psi_m(r))\eta(r) \right)rdr = 0
$$
for any $\eta \in C^{\infty}_c((0,+\infty))$, i.e., $\psi_m(r)$ is a stationary harmonic map profile, which must be smooth on $(0,+\infty)$ and solve the classical problem 
\begin{align}
     0 = \psi_{rr} + \frac{\psi_r}{r} - D^2 \frac{\sin(2\psi)}{2r^2}, \quad r \in (0,+\infty), \label{time independent ode for H v2}
\end{align}
by elliptic regularity. As $v$ and $v_k$ have finite energy, so does $\psi_m(r)$, which implies that $\psi_m(r)$ extends continuously at zero with $\psi_m(0) \in \pi \mathbb Z$.  All such solutions are of the form (\ref{all possible form for limiting H}).

 Indeed, change variables $r = e^s$. Then $\tilde{\psi}(s) = \psi(e^s)$ solves
 \begin{align}
     0 = \tilde{\psi}_{ss} - D^2 \frac{\sin(2\tilde{\psi})}{2}, \quad s \in (-\infty,+\infty).
\end{align}
 In particular,
$$
     0 = 2\tilde{\psi}_{ss} \tilde{\psi}_{s} - 2D^2 \sin(\tilde{\psi}) \cos(\tilde{\psi}) \tilde{\psi}_s, \quad s \in (-\infty,+\infty),
 $$
 meaning that
 $$
 0 = \frac{d}{ds}[\tilde{\psi}_{s}^2] - D^2 \frac{d}{ds}[ \sin(\tilde{\psi})^2], \quad s \in (-\infty,+\infty),
 $$
which integrates to
 $$
C = (r\psi_r)^2 - D^2  \sin(\psi)^2, \quad r \in (0,+\infty).
 $$
As $\psi$ has finite energy, 
$$
\frac{C}{r} =  \left(\psi_r^2 - \frac{D^2}{r^2}  \sin(\psi)^2 \right) r\in L^{1}([0,1],dr)
$$
meaning that $C = 0$. Then, one solves the equation using separation of variables in order to obtain the solutions (\ref{all possible form for limiting H}).

In particular, there is a maximal number of non-constant bubbles as they carry a lower-bounded amount of energy, i.e., $\psi_M = 0$ for all $M$ large enough. 
\end{proof}

In the following, take $M$ minimal such that $\psi_{M+i} = 0$ for all $i \geq 1$.

 We carry on with a few energy facts proving that there is no concentration of energy at the self-similar scale.
\begin{proposition}[No concentration at self-similar scale]
    One has
    \begin{align}
        \lim \limits_{t \to T}E(v(t,\cdot)-v(T,\cdot);B_{1/2}(0) \setminus B_{\alpha \sqrt{T-t}}(0)) &= 0, \quad \alpha \in (0,1], \label{eq: no concentration at self-similar scale}\\
    \lim \limits_{t \to T}||v(t,\cdot)-v(T,\cdot)||_{H^1(B_{1/2}(0) \setminus B_{\alpha \sqrt{T-t}}(0))} &= 0, \quad \alpha \in (0,1].\notag
\end{align}

\end{proposition}

\begin{proof}
    Fix $\alpha \in (0,1]$ and $\eta > 0$. Let $0 < \varepsilon < 1/4$ be such that
    $$
E(v(T,\cdot);B_{2\varepsilon}(0)) \leq \eta.
    $$
Assume that $[a/2,2b] \subset [0,1/2)$, $0 \leq \chi \leq 1$ and $\chi \in C^{\infty}_c(B_{2b}(0) \setminus B_{a/2}(0))$ is $1$ on $[a,b]$, $|\chi_r(r)| \lesssim \max\{a^{-1},b^{-1}\}$ if $a > 0$ and  $|\chi_r(r)| \lesssim b^{-1}$ if $a = 0$. Observe that
$$
 E(v(t);[a,b])  \leq \frac{1}{2} \int_{B_{2b} \setminus B_{a/2} }  |\nabla \Phi(x,t)\chi(x)|^2 dx \leq E(v(t);[a/2,2b]), \quad t \in (0,T). 
$$
In particular,
\begin{align}
E(v(t_2);[a,b]) - E(v(t_1);[a/2,2b]) &\leq \frac{1}{2} \int_{B_{2b} \setminus B_{a/2} }  |\nabla \Phi(x,t)\chi(x)|^2 dx \bigg\rvert^{t = t_2}_{t = t_1} \notag \\
        &= \mathrm{sgn}(t_2-t_1)\int_{[t_1,t_2]} \int_{B_{2b} \setminus B_{a/2} } \frac{1}{2}\frac{d}{dt} \left( |\nabla \Phi|^2 \chi^2 \right) dx dt \notag 
\end{align}
no matter whether $t_2 \geq t_1$ or $t_1 \geq t_2$. Here, $\int_{[t_1,t_2]}$ is understood as the Lebesgue integral on $[\min\{t_1,t_2\},\max\{t_1,t_2\}]$. Going back to (\ref{eq:equality involving v_t}) and using Cauchy-Schwarz inequality with the integrand $2 \chi (\nabla \chi)^T \nabla \Phi \Phi_t$, we obtain that
\begin{align}
&E(v(t_2);[a,b]) - E(v(t_1);[a/2,2b]) 
    \lesssim  \int_{[t_1,t_2]} \int_{0}^{1/2} |v_t|^2rdrdt  \notag \\
    &+ \max\{1_{a > 0}a^{-1},b^{-1}\}  \cdot \sqrt{|t_2-t_1|} \sup_{t \in [0,T]} E(v;[0,1])^{1/2} \cdot \left( \int_{[t_1,t_2]}\int_{0}^{1/2} |v_t|^2 rdrdt \right)^{1/2}. \label{eq: better estimate for no concentration at self similar scale}
\end{align}

Apply (\ref{eq: better estimate for no concentration at self similar scale}) with fixed $t_2 \in (0,T)$, $t_1 \to T$, $[a,b] = [\alpha \sqrt{T-t_2}, \varepsilon]$. It yields that
\begin{align*}
    &E(v(t_2,\cdot);B_{\varepsilon}(0) \setminus B_{\alpha \sqrt{T-t_2}}(0))- E(v(T,\cdot);B_{2\varepsilon}(0) \setminus B_{\alpha \sqrt{T-t_2}/2}(0))
\\
&\lesssim \int_{t_2}^{T} \int_{0}^{1/2} |v_t|^2rdrdt + \left( \int_{t_2}^T \int_{0}^{1/2} |v_t|^2rdrdt \right)^{1/2}.
\end{align*}
Take the limsup as $t_2 \to T$ to conclude that 
\begin{equation*}
 \limsup_{t \to T} E(v(t,\cdot);B_{\varepsilon}(0) \setminus B_{\alpha \sqrt{T-t}}(0)) \leq  E(v(T,\cdot);B_{2\varepsilon}(0)) \leq \eta.
\end{equation*}
By triangle inequality, 
\begin{equation}
\limsup_{t \to T} E(v(t,\cdot)-v(T,\cdot);B_{\varepsilon}(0) \setminus B_{\alpha \sqrt{T-t}}(0)) \lesssim 2 \eta.\label{eq: no concentration at self similar scale, inner region}
\end{equation} 

On the region $B_{1/2}(0) \setminus B_{\varepsilon}(0)$, $v(r,t)$ is smooth up to $T$. Hence, 
\begin{equation}
 \limsup_{t \to T}  E(v(t,\cdot)-v(T,\cdot);B_{1/2}(0) \setminus B_{\varepsilon}(0)) = 0. \label{eq: no concentration at self similar scale, outer region}
\end{equation}
Combine (\ref{eq: no concentration at self similar scale, inner region}) and (\ref{eq: no concentration at self similar scale, outer region}), then let $\eta \to 0$ to conclude. This directly implies convergence to zero for the $\dot{H}^1$-norm as well. Convergence in the $L^2$-norm follows by dominated convergence ($v$ is bounded and the domain of integration is bounded).
\end{proof}

\begin{corollary}\label{cor: lambda_k ll sqrt T-t_k}
    One actually has $\lambda_{m,k} \ll \sqrt{T-t_k}$.
\end{corollary}

\begin{proof}
    Assume not. Then, up to passing to a subsequence,
    $$
    \frac{\lambda_{m,k}}{\sqrt{T-t_k}} \to c > 0
    $$
    by Lemma \ref{lemma: lambda_k lessim sqrt T-t_k}. As the profile $\psi_m$ is non-constant, there is $0 < a < b < +\infty$ for which
    $$
\int_a^b |\psi_m'(r)|^2 rdr > 0.
    $$
    As $\psi_m(r)$ is the weak $\dot{H}^1(B_b(0) \setminus B_a(0)))$-limit of $v(t_k, \lambda_{m,k}\rho)$, one obtains
  \begin{align*}
      0 < \int_a^b |\psi_m'(r)|^2 rdr &\leq \liminf_{k \to +\infty} \int_a^b |\lambda_{m,k} \partial_r v(t_k, \lambda_{m,k}\rho)|^2 \rho d\rho \\
&\leq \liminf_{k \to +\infty} \int_{a\lambda_{m,k}}^{b\lambda_{m,k}} | \partial_r v(t_k, s)|^2 sds.
  \end{align*}
   Since $v(T)$ has finite energy,
   $$
\lim \limits_{k \to +\infty} \int_{a\lambda_{m,k}}^{b\lambda_{m,k}} | \partial_r v(T, s)|^2 sds = 0
   $$
   by absolute continuity of the Lebesgue integral, hence
   $$
   \liminf_{k \to +\infty} \int_{a\lambda_{m,k}}^{b\lambda_{m,k}} | \partial_r v(t_k, s)-\partial_r v(T,s)|^2 sds > 0.
   $$
    But for $k$ large enough, one has
    $$
    [a \lambda_{m,k}, b\lambda_{m,k}] \subset [\alpha \sqrt{T-t_k}, 1/4]
    $$
    for some $\alpha < \min\{1,(ac)/2\}$. This contradicts (\ref{eq: no concentration at self-similar scale}).
\end{proof}

\begin{corollary}[No concentration at self-similar scale]\label{cor:no concentration at self-similar scale, epsilon version}
        One has
    \begin{align}
        \lim \limits_{k \to +\infty}E(\varepsilon_{M,k};B_{1/4}(0) \setminus B_{\alpha \sqrt{T-t_k}}(0)) &= 0, \quad \alpha \in (0,1].\label{eq: no concentration at self-similar scale for the remainder}
    \end{align}
\end{corollary}

\begin{proof}
    This follows from the profile decomposition, together with the triangle inequality for the energy, (\ref{eq: no concentration at self-similar scale}) and
\begin{align*}
        \lim \limits_{k \to +\infty}E(\psi_m(r/\lambda_{m,k});B_{1/4}(0) \setminus B_{\alpha \sqrt{T-t_k}}(0)) &= 0, \quad \alpha \in (0,1],\notag \\
    \lim \limits_{k \to +\infty}||\psi_m(r/\lambda_{m,k})||_{\dot{H}^1(B_{1/4}(0) \setminus B_{\alpha \sqrt{T-t_k}}(0))} &= 0, \quad \alpha \in (0,1],
\end{align*}
    because $\lambda_{m,k} \ll \sqrt{T-t_k}$ for any $m = 1, ..., M$ by Corollary \ref{cor: lambda_k ll sqrt T-t_k}.
\end{proof}

\subsection{Strong convergence of the remainder}
Next, we prove that the scaling symmetry explains all the energy loss of $v(t)$ when going from $t < T$ to time $T$. For this, we show that the remainder $\varepsilon_{M,k}$ of our profile decomposition converges strongly to zero as $k \to +\infty$.

 \begin{lemma}[Small energy implies proximity to a multiple of $\pi$] \label{lemma:closeness to l0 pi}
 For all $\pi/2 > \varepsilon > 0$, $R_0 > 1$, there exists $\eta > 0$ such that for all $0 \leq R_1 < R_2 \leq 1$, $R_2/R_1 \geq R_0$, $v(r)$ with $E(v;[R_1,R_2]) < \eta$, there is a unique $l_0 \in \mathbb Z$ for which $|v(r)-l_0\pi| < \varepsilon$ on $[R_1,R_2]$. If, in addition, $\varepsilon \leq \varepsilon_0 \leq \pi/2$ for some $\varepsilon_0$ small enough, then $\mathcal{E}(v-l_0\pi;[R_1,R_2]) \leq 2E(v;[R_1,R_2])$, where
 $$
E(v;[R_1,R_2]) := \int_{R_1}^{R_2} \left( |v_r|^2 + \frac{D^2}{r^2}|\sin(v)|^2 \right) rdr, \quad \mathcal{E}(v;[R_1,R_2]) := \int_{R_1}^{R_2} \left( |v_r|^2 + \frac{D^2}{r^2}|v|^2 \right) rdr.
 $$
 \end{lemma}
 
 \begin{proof}
 See \cite{jendrej}, Lemma 2.3.
 \end{proof}

\begin{lemma}\label{lemma:extraction of scale mu_k}
    If $$
\liminf \limits_{k \to +\infty} E(\varepsilon_{M,k};[0,1/4]) > 0,
$$
then there exists $0 < \mu_k \leq  \sqrt{T-t_k}/4$, such that
$$
\liminf \limits_{k \to +\infty} E(\varepsilon_{M,k};[\mu_k/2,\mu_k]) > 0.
$$
\end{lemma}

\begin{proof}
    Fix $\delta = 1/4$. Assume for a contradiction that
    $$
\lim \limits_{k \to +\infty} \sup_{0 < \mu < \delta \sqrt{T-t_k}}  E(\varepsilon_{M,k};[\mu/2,\mu]) =  0.
    $$
    Apply Lemma \ref{lemma:closeness to l0 pi} with $0< \tilde{\varepsilon} < \pi/4$ and $R_0 = 2$ on dyadic intervals $I_{j,k} = [2^{-j-1}\delta\sqrt{T-t_k}, 2^{-j}\delta\sqrt{T-t_k}]$, $j \geq 0$.
 Let $k_{\tilde{\varepsilon}}$ be large enough so that
    $$\sup_{0 < \mu < \delta\sqrt{T-t_k}}  E(\varepsilon_{M,k};[\mu/2,\mu]) \leq \eta_{\tilde{\varepsilon}}, k \geq k_{\tilde{\varepsilon}},$$
    where $\eta_{\tilde{\varepsilon}}$ is given by Lemma \ref{lemma:closeness to l0 pi}. Then 
   there exists $l_{k,j} \in \mathbb Z$ for which
    $$
    |\varepsilon_{M,k}(r) - l_{k,j} \pi| \leq \pi/4, \quad r \in I_{k,j}.  
    $$
    Looking at the common endpoint of $I_{j,k}$ and $I_{j+1,k}$, one deduces that
    $$
|l_{k,j}\pi - l_{k,j+1}\pi| \leq \pi/2
    $$
    which implies $l_{k,j} = l_{k,j+1} = l_k \in \mathbb Z$ does not depend on $j$ (for a similar reason, the $l_{k}$ do not depend on $\tilde{\varepsilon}$). Thus, we have proved that
    $$
        |\varepsilon_{M,k}(r) - l_{k} \pi| \leq \tilde{\varepsilon}, \quad r \in I_{k} := [0,\delta\sqrt{T-t_k}], k \geq k_{\tilde{\varepsilon}}. 
        $$
        Setting $\tilde{\varepsilon} = 1/n$ and passing to the subsequence $k_{1/n}$ (which we still denote using index $k$), we find that
        $$
\lim \limits_{k \to +\infty} \sup_{0 < r < \delta \sqrt{T-t_k}}  |\varepsilon_{M,k}(r) - l_{k} \pi| = 0.
        $$
        As $v$ and the profiles $\psi_m$ are bounded, we deduce that $l_k \pi$ is bounded uniformly with respect to $k$. Up to passing to a subsequence, we can assume that $l_k = l$ is constant. In other words,
        $$
\lim \limits_{k \to +\infty} ||\varepsilon_{M,k}-l\pi||_{L^{\infty}([0,\delta\sqrt{T-t_k}])} = 0.
        $$
        We will use this smallness assumption to obtain that
\begin{equation}
            \lim_{k \to +\infty} E(\varepsilon_{M,k};[0,\delta \sqrt{T-t_k}]) = 0, \label{eq: claim energy contradiction}
\end{equation}
        which is a contradiction

        Let $e_k(r) = \varepsilon_{M,k}(r) - l\pi$, $\rho_k = \sqrt{T-t_k}, \eta \in C^{\infty}_c([0,\delta)), \eta_k(r) = \eta(r/\rho_k)$. For $k$ large enough, the support of $\eta_k e_k$ is a subset of $[0,1/4)$. 
        As in (\ref{eq:palais smale condition}),
        $$
|E'(v(t_k))[e_k \eta_k]| \lesssim \sqrt{T-t_k} ||v_t(t_k)||_{L^2([0,1/4],rdr)} \cdot ||\eta||_{L^2((0,+\infty),rdr)} \cdot ||e_k||_{L^{\infty}([0,\rho_k])} \to 0.
$$
After integrating by parts, $E'(v(t_k))[e_k \eta_k] = o_k(1)$ rewrites as
$$2\pi\int_0^{\infty} \left( v_k'(r)[\eta'(r)R_k(r) +\eta(r) R_k'(r)] + \frac{D^2}{r^2}\sin(v_k(r))\cos(v_k(r))\eta(r)R_k(r) \right)rdr
$$        
where $v_k(r) = v(\rho_kr,t_k) = W_k(r) + R_k(r)$, $R_k(r) = e_k(\rho_k r)$. As $v$ is smooth at $r = 0, t = t_k$ (e.g., \cite[Proposition 3.8]{samuelian2025blowuptrees}) and $\eta_ke_k$ is bounded, the boundary term
$$
 2\pi[rv_k'(r)\eta_k(r)e_k(r)]_{r=0}  = 0
 $$
 vanished.

As $e_k$ is small and $v$ has finite energy, observe that
$$
\left| \int_0^{\infty} v_k'(r)\eta'(r)R_k(r) rdr\right| \leq ||v_r(\cdot,t_k)||_{L^2(rdr)} \cdot ||\eta'||_{L^2(rdr)} \cdot ||e_k||_{L^{\infty}([0,\delta\rho_k])} \to 0.
 $$
Similarly,
\begin{align*}
(NL) :  &\int_0^{\infty}\frac{1}{r^2}\left[ \sin(W_k+R_k)\cos(W_k+R_k) - \sin(R_k)\cos(R_k) \right] \eta(r)R_k(r) rdr \to 0,
\end{align*}
Indeed, write $$
W_k(r) = v(T,\rho_k r) + \sum_{m=1}^M \psi_m \left(\frac{\rho_k r}{\lambda_{m,k}} \right) + l \pi
$$
and $\lim \limits_{r \to 0} v(r,T) = c\pi$. Let
$$
d(x) := \mathrm{dist}(x,\pi \mathbb Z).
$$
As $s \mapsto \sin(s)\cos(s)$ is uniformly Lipschitz (its derivative is $\cos(2s)$) and $\pi$-periodic,
\begin{align*}
    |(NL)| &\lesssim \int_0^{\delta}\frac{1}{r^2}|d(W_k)\eta(r)R_k(r)|rdr  \\
    \int_0^{\delta}\frac{1}{r^2}|d(v(T,\rho_kr))\eta(r)R_k(r)|rdr  &\lesssim \int_0^{\delta}\frac{1}{r^2}|\sin(v(T,\rho_k r) - c \pi)\eta(r)R_k(r)|rdr \\
    &\lesssim E(v(T);[0,\delta\rho_k])^{1/2} \cdot E(e_k;[0,\delta\rho_k])^{1/2}  \cdot ||\eta||_{L^{\infty}} \to 0,
\end{align*} 
where we used that $R_k^2 \lesssim \sin(R_k)^2$, while
\begin{align*}
     &\int_0^{\delta}\frac{1}{r^2}\left| d\left[\psi_m\left( \frac{\rho_k r}{\lambda_{m,k}} \right)\right] \eta(r)R_k(r)\right|rdr \\
     &\lesssim ||d(\psi_m(r))||_{L^1(r^{-1}dr)} \cdot ||\eta||_{L^{\infty}} \cdot ||e_k||_{L^{\infty}([0,\delta\rho_k])} \to 0.
\end{align*}
 Hence, it follows that
$$
2\pi \int_0^{\infty} \left( v_k'(r)\eta(r)R_k'(r) + \frac{D^2}{r^2}\sin(R_k(r))\cos(R_k(r))\eta(r) R_k(r)\right)rdr \to 0.
$$
Assume that $\eta = 1$ on $[0,a] \subset [0,\delta)$ and observe that
\begin{align*}
    2\pi &\int_a^{\delta} \left( |v_k'(r)R_k'(r)| + \frac{D^2}{r^2}|\sin(R_k(r))\cos(R_k(r))R_k(r)|\right)rdr \\
    =  2\pi &\int_{a\rho_k}^{\delta\rho_k} \left( |v_r(r,t_k)e_k'(r)| + \frac{D^2}{r^2}|\sin(e_k(r))\cos(e_k(r))e_k(r)|\right)rdr \\
    \lesssim 2\pi&\int_{a\rho_k}^{\delta\rho_k} \left( |v_r(r,t_k)e_k'(r)| + \frac{D^2}{r^2}\sin(e_k(r))^2\right)rdr \\
    \lesssim 2\pi& E(e_k,B_{\delta}(0) \setminus B_{a \rho_k}(0))^{1/2} E(v(t_k),B_{\delta}(0))^{1/2} + 2\pi E(e_k,B_{\delta}(0) \setminus B_{a \rho_k}(0)) \to 0.
\end{align*}
thanks to Corollary \ref{cor:no concentration at self-similar scale, epsilon version}. Hence, for any fixed $0 < a < \delta$,
$$
2\pi \int_0^{a} \left( v_k'(r)R_k'(r) + \frac{D^2}{r^2}\sin(R_k(r))\cos(R_k(r)) R_k(r)\right)rdr \to 0.
$$
Note that
$$
(L) : 2\pi \int_0^{a} W_k'(r)R_k'(r)rdr \to 0.
$$
as well. Indeed,
$$
\left| \int_0^{a} \partial_r[v(\rho_kr,T)]R_k'(r) rdr \right| \leq ||v_r(s,T)||_{L^2([0,a\rho_k],rdr)} \cdot ||R_k'||_{L^2(rdr)}  \to 0
$$
by absolute continuity of the Lebesgue integral, while
\begin{align*}
   \int_0^{a} \partial_r[\psi_m(\rho_k \lambda_{m,k}^{-1}r)]R_k'(r) rdr  &= \int_0^{\infty} \psi_m'(r)\partial_r[e_k(\lambda_{m,k}r)] rdr \\
   &- \int_{a \rho_k \lambda_{m,k}^{-1}}^{\infty} \psi_m'(r)\partial_r[e_k(\lambda_{m,k}r)] rdr = (I) + (II)
   \end{align*}
Then $(I) \to 0$ using the weak $\dot{H}^1(\mathbb R^2)$-convergence of $\varepsilon_{M,k}(\lambda_{m,k}\cdot)$ to zero and 
$$
|(II)| \lesssim ||\psi_m'||_{L^2([a \rho_k \lambda_{m,k}^{-1},\infty),rdr)} \cdot E(\varepsilon_{M,k};[0,\delta])^{1/2} \to 0.
$$
In other words, for any fixed $0 < a < \delta$,
$$
2\pi \int_0^{a} \left( |R_k'(r)|^2 + \frac{D^2}{r^2}\sin(R_k(r))\cos(R_k(r)) R_k(r)\right)rdr \to 0.
$$
As $e_k$ is small, 
$$
\sin(R_k(r))\cos(R_k(r)) R_k(r) \geq \frac{1}{2}\sin(R_k(r))^2,
$$
which leads to the claim (\ref{eq: claim energy contradiction}) on $[0,a\sqrt{T-t_k}]$ after changing variables. On $[a\sqrt{T-t_k},\delta\sqrt{T-t_k}]$, use Corollary \ref{cor:no concentration at self-similar scale, epsilon version}.
\end{proof}

\begin{proposition}[Strong convergence of the remainder]\label{prop:strong convergence of remainder}
The remainder $\varepsilon_{M,k}$ of the profile decomposition converges strongly to zero in energy norm, i.e.
$$
\lim \limits_{k \to +\infty} E(\varepsilon_{M,k};[0,1/4]) = 0.
$$
In particular, the convergence holds in $\dot{H}^1(B_{1/4}(0))$. Moreover, 
$$
\lim \limits_{k \to +\infty}||\varepsilon_{M,k}(r)  + \sum_{m=1}^M \psi_m(\infty)||_{H^1(B_{1/4}(0))} = 0.
$$
\end{proposition}

\begin{proof}
If $\varepsilon_{M,k}$ does not converge to zero in energy norm, then one can extract a scale $\mu_k \lesssim  \sqrt{T-t_k}$ along which
\begin{align}
    \liminf_{k \to +\infty} E(\varepsilon_{M,k},[\mu_k/2,\mu_k]) = \liminf_{k \to +\infty}  E(R_k,[1/2,1]) > 0 \label{eq: weak limit R_k non constant}
\end{align}
by Lemma \ref{lemma:extraction of scale mu_k}, where $R_k = \varepsilon_{M,k}(\mu_k r)$. Up to taking a further subsequence, $R_k(r)$ has a weak $\dot{H}^1(\mathbb R^2)$-limit $R_{\infty}$. If $R_{\infty}$ is non-constant, then one would be able to extract a non-constant profile $\psi_{M+1}$, according to the criterion (\ref{eq: finiteness criterion decompo}), which would contradict the maximality of our profile decomposition. Hence, $R_{\infty}$ is constant.

We upgrade the weak $\dot{H}^1$ convergence of $R_k$ to a weak $H^{2}(B_1 \setminus B_{1/2})$-convergence, which implies strong $H^{1}(B_1 \setminus B_{1/2})$-convergence by Rellich-Kondrachov Theorem and uniform convergence thanks to the embeddings $H^{2}(B_1 \setminus B_{1/2}) \hookrightarrow C^{0,\alpha}(B_1 \setminus B_{1/2}) \subset \subset C^0(B_1 \setminus B_{1/2})$. Consider $V_k(r) = v(\mu_kr,t_k)$, which solves
$$
V_k''+\frac{V_k'}{r}- \frac{D^2}{2r^2}\sin(2V_k)=\mu_k^2 v_t(\mu_kr,t_k), \quad r \in [1/2,1].
$$
One has
$$
||\mu_k^2v_t(\mu_kr,t_k)||_{L^2([1/2,1],rdr)} \lesssim \sqrt{T-t_k}||v_t(t_k)||_{L^2([0,1/4],rdr)}  \to 0,
$$
and $V_k'$ is uniformly bounded in $L^2([1/2,1],rdr)$. Hence, $V_k''$ is uniformly bounded in $L^2([1/2,1],rdr)$, which leads to the weak $H^{2}(B_1 \setminus B_{1/2})$-convergence, up to passing to a subsequence. Then 
$$
R_k = V_k - W_k =v(\mu_k r,t_k) - \left[ v(\mu_k r,T)
+\sum_{m=1}^{M}\psi_m\left(\frac{\mu_k r}{\lambda_{m,k}}\right) \right]
$$
also converges $H^1$-strongly and uniformly on the annulus $B_1 \setminus B_{1/2}$. Then 
$$
  E(R_{\infty},[1/2,1]) = \lim_{k \to +\infty}  E(R_k,[1/2,1]) > 0 
  $$
implies that the weak limit $R_{\infty}$ of $R_k$ is not a multiple of $\pi$. Yet, by Fatou's Lemma, $R_{\infty}$ has finite energy on $[0,+\infty)$, meaning that it must be a multiple of $\pi$. We have found a contradiction.

 This finishes the proof that the remainder $\varepsilon_{M,k}(r)$ converges strongly to zero in energy norm. In fact, $\varepsilon_{M,k}(r)$ converges strongly to 
$$
-\sum_{m=1}^M \psi_m(\infty).
$$
in $H^1(B_{1/4}(0))$. Indeed, $\varepsilon_{M,k}$ is bounded in $H^1(B_{1/4}(0))$, hence it converges weakly up to taking a further subsequence. Moreover, $v(r,t_k)$ converges weakly in $H^1((0,1))$ to $v(r,T)$, while the rescaled profiles converge weakly to
$$
\psi_m\left( \frac{r}{\lambda_{m,k}} \right) \to \psi_m(\infty).
$$
Hence, the weak limit $H^1(B_{1/4}(0))$-limit of $\varepsilon_{M,k}(r)$ is $-\sum_{m=1}^M \psi_m(\infty)$. In particular,
$$
m_k = \frac{1}{|B_{1/4}(0)|} \int_{B_{1/4}(0)} \varepsilon_{M,k}(r)dr = \frac{1}{|B_{1/4}(0)|} \langle \varepsilon_{M,k}, 1 \rangle_{L^2(B_{1/4}(0))} \to -\sum_{m=1}^M \psi_m(\infty) 
$$
by weak $L^2$-convergence. The strong convergence then follows from Poincaré inequality. 
\end{proof}

As $M$ is fixed, we now omit the $M$ in the remainder's notation. We have proved that
$$
  ||\tilde{\varepsilon}_k ||_{H^1(B_{1/4}(0))} := ||v(r,t_k) - v(r,T) - \sum_{m=1}^M [\psi_{m} \left(\frac{r}{\lambda_{m,k}} \right) - \psi_{m}(\infty) ] ||_{H^1(B_{1/4}(0))} \to 0,
$$
as well as in energy norm.

Finally, under the normalization $v(0,t) = 0$, $t < T$, the decomposition for $v(r,t)$ yields a decomposition for $u(r,t) = r^{-D}v(r,t)$ after multiplying everything by $r^{-D}$,
\begin{equation}
     u(r,t_k) = r^{-D}\left[ v(r,T) + \sum_{m=1}^M [\psi_{m} \left(\frac{r}{\lambda_{m,k}} \right) - \psi_{m}(\infty) ]  + \tilde{\varepsilon}_{k}(r) \right]. \label{eq: decomposition u in R^2D+2}
\end{equation}

\begin{lemma}\label{lemma:equivalence of remainder convergence}
    Let $ \tilde{\varepsilon} \in H^1(B^{(2)}_{\delta}(0))$ with finite energy $E(\tilde{\varepsilon};B^{(2)}_{\delta}(0))  <  \eta(\varepsilon_0)$, where $\eta$ is as in Lemma \ref{lemma:closeness to l0 pi} applied with $\varepsilon_0 \leq \pi/2$, and assume $\tilde{\varepsilon}(0) = 0$. There exists $C(D,\delta)$ for which
    $$
||r^{-D} \tilde{\varepsilon}||_{H^1(B^{(2D+2)}_{\delta}(0))} \leq C(D,\delta) \left[|| \tilde{\varepsilon}||_{H^1(B^{(2)}_{\delta}(0))}+E(\tilde{\varepsilon};B^{(2)}_{\delta}(0))^{1/2} \right].
    $$
    Conversely, let $\varepsilon \in H^1(B^{(2D+2)}_{\delta}(0))$. There exists $C(D,\delta)$ for which
    $$
||r^D\varepsilon||_{H^1(B^{(2)}_{\delta}(0))} +\mathcal{E}(r^D\varepsilon;B_{\delta}^{(2)}(0))^{1/2} \leq C(D,\delta)|| \varepsilon||_{H^1(B^{(2D+2)}_{\delta}(0))}.
    $$
\end{lemma}

\begin{proof}
    Consider first $\tilde{\varepsilon}$. The $L^2$-norms for $\tilde{\varepsilon}$ and $r^{-D}\tilde{\varepsilon}$ are equivalent. As for the $\dot{H}^1(B_{\delta}(0) \subset \mathbb R^{2D+2})$-norm, one checks that
\begin{align*}
    ||r^{-D} \tilde{\varepsilon}(r)||_{\dot{H}^1(B_{\delta}^{(2D+2)}(0))}^2 &\sim \int_0^{\delta}|\tilde{\varepsilon}'(r) - \frac{D}{r}\tilde{\varepsilon}(r)|^2rdr \\
    &\leq 2||\tilde{\varepsilon}||_{\dot{H}^1(B^{(2)}_{\delta}(0))}^2  + 2 \int_0^{\delta}\left| \frac{D}{r}\tilde{\varepsilon}(r)\right|^2rdr.
\end{align*}
Apply Lemma \ref{lemma:closeness to l0 pi}, where one must have $l_0 = 0$ (the finite energy assumption implies that $\tilde{\varepsilon}(r)$ is continuous at zero with $\tilde{\varepsilon}(0) = 0$).

Conversely, consider $\varepsilon$. Again, the main difficulty concerns only the $\dot{H}^1$-norm. One checks that
\begin{align*}
    ||r^{D} \varepsilon(r)||_{\dot{H}^1(B_{\delta}^{(2)}(0))}^2 &\sim  \int_0^{\delta}|\varepsilon'(r) + \frac{D}{r}\varepsilon(r)|^2r^{2D+1}dr \\
    &\leq 2||\varepsilon||_{\dot{H}^1(B_{\delta}^{(2D+2)}(0))}^2  + 2\int_0^{\delta}\left| \frac{D}{r}\varepsilon(r)\right|^2r^{2D+1}dr.
\end{align*}
Split the integral into two parts by multiplying the inside with a smooth cut-off $\chi_{\delta/2}$. The part near the origin $r \lesssim \delta$ is bounded using the $(2D+2)$-dimensional Hardy's inequality (\ref{eq: rellich inequ}) and the part near the boundary $r \sim \delta$ is bounded using the $L^2$-norm of $\varepsilon$. The same can be done with $\mathcal{E}(r^D \varepsilon;[0,\delta])$.
\end{proof}

\begin{proposition}\label{prop:strong convergence of tilde epsilon}
The remainder of the decomposition (\ref{eq: decomposition u in R^2D+2}) also converges strongly to zero, in the sense that $\tilde{\varepsilon}_k(0) = 0$ for all $k$ and
$$
 ||r^{-D}\tilde{\varepsilon}_k||_{H^1(B^{(2D+2)}_{1/4}(0))} \to 0.
$$ 
\end{proposition}

\begin{proof}
The proposition is a consequence of Lemma \ref{lemma:equivalence of remainder convergence} if one can prove $\tilde{\varepsilon}_k(0) = 0$. Note that $\tilde{\varepsilon}_k(r)$ has finite energy, hence $\tilde{\varepsilon}_k(0) = c_k \pi, c_k \in \mathbb Z$. Since $v(0,t_k) = 0$, evaluating the decomposition at $r = 0$ shows that $c_k = c$ is constant. Assume $c \neq 0$ for a contradiction. As $\tilde{\varepsilon}_k \to 0$ in $H^1_{loc}((0,1/4))$, it converges uniformly to zero away from the origin. Pick $1/4 > r_0 > 0$ such that $\tilde{\varepsilon}_k(r_0) \to 0$. As $a^2 + b^2 \geq 2|ab|$ and 
$$
\partial_r \int_0^{\tilde{\varepsilon}_k(r)}|\sin(s)|ds = \tilde{\varepsilon}_k' |\sin(\tilde{\varepsilon}_k)|,
$$
we find that
\begin{align*}
    \int_0^{r_0}r(\tilde{\varepsilon}_k')^2 + \frac{D^2}{r}\sin(\tilde{\varepsilon}_k)^2dr &\geq D\int_0^{r_0} |\tilde{\varepsilon}_k'\sin(\tilde{\varepsilon}_k)|dr \\
    &\geq D \left| \int_0^{r_0} \tilde{\varepsilon}_k'|\sin(\tilde{\varepsilon}_k)| dr \right|  \\
    &= D\left| \int_{\tilde{\varepsilon}_k(r_0)}^{\tilde{\varepsilon}_k(0)}|\sin(s)|ds \right| \\
    &\to D\left|  \int_0^{c \pi}|\sin(s)|ds  \right| > 0,
\end{align*}
which contradicts the fact that $\tilde{\varepsilon}_k$ converges to zero in energy norm. 
\end{proof}

\begin{corollary}[On the value of $v(0,T)$]
In particular, evaluating the decomposition (\ref{eq: decomposition u in R^2D+2}) at $r = 0$ yields 
$$
0 = v(0,T) + \sum_{m=1}^M [\psi_m(0) - \psi_m(\infty)].
$$
\end{corollary}

The work of Samuelian (\cite{samuelian2025blowuptrees}) shows that there can be at most one non-constant profile in the decomposition for $v$ and $u$, i.e. $\psi_M = 0$ for $M \geq 2$. Using the normalization $v(0,t) = 0$, 
$t < T$, and up to a change of sign, one obtains $v(0,T) = \pi$ and $\psi_1(r) = 2\arctan(r^D)$, $\psi(\infty) = \pi$ up to replacing $\lambda_{1,k}$ by $\lambda_{1,k}/\alpha$ for some $\alpha > 0$. Excluding the existence of bubbles shows global existence. Indeed, if $M = 0$, 
$$
0 = \tilde{\varepsilon}_k(0) = v(0,t_k) - v(0,T) = 0 - v(0,T) 
$$
by Proposition \ref{prop:strong convergence of tilde epsilon}, which forces $v(0,T) = 0$  and prevents blow-up at $T < +\infty$.

Using the single bubble result, we will write
\begin{equation}
     u(r,t_k) = \underbrace{r^{-D}[ v(r,T) - \pi]}_{:= u_*(r)} + 2r^{-D}\arctan\left( \frac{r^D}{\lambda_k^D}\right)  + r^{-D}\tilde{\varepsilon}_k(r). \label{eq: decomposition for u in R^2D+1, single bubble case}
\end{equation}

We finish this section with some other energy facts.

\begin{lemma}[Limit of energy exists]
    The limit
    \begin{align}
       \lim \limits_{t \to T}E(v(t,\cdot);B_{\alpha \sqrt{T-t}}(0)) = E(Q), \quad \alpha \in (0,1/2], \label{eq: energy of bubble is recovered on 0 < r < sqrt T - t}
\end{align}
exists.
\end{lemma}

\begin{proof}
  From the profile decomposition, Proposition \ref{prop:strong convergence of remainder} and the triangle inequality for the energy, one deduces that
  $$
   \lim \limits_{k \to +\infty}E(v(t_k,\cdot);B_{\alpha \sqrt{T-t_k}}(0))^{1/2} = E(Q)^{1/2}, \quad \alpha \in (0,1],
   $$
   exists. Moreover,
     $$
   \lim \limits_{k \to +\infty}\left|E(v(t_k,\cdot);B_{\alpha \sqrt{T-t}}(0))^{1/2} - E(Q)^{1/2}\right| \leq E(v(T);B_{\alpha \sqrt{T-t}}(0))^{1/2}=  o_{t \to T}(1).
   $$
   
   Apply (\ref{eq: better estimate for no concentration at self similar scale}) with fixed $t_1 \in (0,T)$, $t_2 \to T$ along the sequence of times $t_k$ from the profile decomposition, $[a,b] = [0,\alpha \sqrt{T-t_1}]$ and $\chi(r) \in C^{\infty}_c([0,2b])$, $\chi(r) = 1$ on $[a,b]$, $|\chi_r(r)| \lesssim (T-t_1)^{-1/2}$. It yields that
\begin{align*}
    o_{t_1 \to T}(1) + E(Q) &- E(v(t_1,\cdot); B_{2\alpha \sqrt{T-t_1}}(0))\\
    &\lesssim \int_{t_1}^{T} \int_{0}^{1/2} |v_t|^2rdrdt + \left( \int_{t_1}^T \int_{0}^{1/2} |v_t|^2rdrdt \right)^{1/2}.
\end{align*}
Take the limit as $t_1 \to T$ to conclude that 
\begin{equation*}
E(Q) - \liminf \limits_{t \to T} E(v(t,\cdot); B_{2\alpha \sqrt{T-t}}(0)) \leq 0, \quad \alpha \in (0,1].
\end{equation*}
Exchanging the roles of $t_1$ and $t_2$ yields 
\begin{equation*}
\limsup \limits_{t \to T} E(v(t,\cdot); B_{\alpha \sqrt{T-t}}(0)) - E(Q) \leq 0, \quad \alpha \in (0,1/2].
\end{equation*}
As $\alpha$ is arbitrary, this concludes the proof.
\end{proof}

\begin{lemma}[Pythagorean Decomposition of Energy]\label{lemma: failure of convergence of v(t) to v(T) in norm}
    For any $0 < \eta < 1/4$,
    \begin{equation}
   \lim_{t \to T} ||v(t,\cdot)-v(T,\cdot)||_{\dot{H}^1(B_{\eta }(0))}^2  = \lim_{t \to T} ||v(t,\cdot)-v(T,\cdot)||_{H^1(B_{\eta }(0))}^2 = E(Q)\label{eq: existence limit of energy at time T}
\end{equation}
    exists.
\end{lemma}

\begin{proof}
It suffices to show that the limit exists. The value of the limit is obtained using the Pythagorean decomposition of the energy (\ref{eq:pythagorean decompo energy profile}) (see also Remark \ref{rmk: pythagorean decomposition on any ball}) and the strong convergence of the remainder (Proposition \ref{prop:strong convergence of remainder}).

Fix $\delta = 1/4$. Recall (\ref{eq: no concentration at self-similar scale}). Let $\sqrt{T-t} \ll \sigma(t) \ll 1$ be a scale defined as $\sigma(t) = e^{-L(t)}$, where $L(t)$ grows slowly enough so that
$$
\lim \limits_{t \to T} L(t) = +\infty, \quad \lim \limits_{t \to T} L(t)||v(t,\cdot)-v(T,\cdot)||_{H^1(B_{\delta}(0) \setminus B_{\alpha \sqrt{T-t}}(0))}^2 = 0, \quad L(t) \ll |\log(T-t)|.
$$
Let $w(t,r) = v(t,r)-v(T,r)$. By Cauchy-Schwarz,
\begin{align*}
    |w(t,\sigma(t)) - w(t,\delta)| &= \left( \int_{\sigma(t)}^{\delta}|w_r|^2rdr \right)^{1/2} \left( \int_{\sigma(t)}^{\delta} r^{-1}dr \right)^{1/2}  \\
    &\leq ||v(t,\cdot)-v(T,\cdot)||_{H^1(B_{\delta}(0) \setminus B_{\alpha \sqrt{T-t}}(0))} \cdot \sqrt{L(t)} \to 0.
\end{align*}
Moreover, $w(t,\delta) \to 0$ as $v$ is smooth at $r = \delta > 0$ up to time $T$. It follows that
$$
\lim \limits_{t \to T}|v(t,\sigma(t))-\pi| \leq \lim \limits_{t \to T} |v(t,\sigma(t)) - v(T,\sigma(t))| = \lim \limits_{t \to T} |w(t,\sigma(t))| = 0
$$
since $\lim \limits_{t \to T} v(T,\sigma(t)) = \pi$. Let
$$
A(t) = \int_0^{\sigma(t)}v_r^2rdr, \quad B(t) =  \int_0^{\sigma(t)} \frac{D^2}{r^2} \sin(v)^2 rdr.
$$
It follows from (\ref{eq: no concentration at self-similar scale}) and (\ref{eq: energy of bubble is recovered on 0 < r < sqrt T - t}) that
$$
\lim_{t \to T} E(v(t,\cdot);B_{\sigma(t)}(0)) = E(Q) = 4\pi D, \quad \lim_{t \to T} (A(t) + B(t)) = 4D.
$$
Cauchy-Schwarz also gives
\begin{align*}
    \sqrt{A(t)B(t)} &\geq D \left| \int_0^{\sigma(t)} v_r \sin v dr \right| \\
    &\geq  D |\cos(v(t,0)) - \cos(v(t,\sigma(t)))| = 2D + o_{t \to T}(1).
\end{align*}
Finally, recall that
$$
2\sqrt{A(t)B(t)} \leq A(t) + B(t),
$$
hence
$$
\lim_{t \to T} 2\sqrt{A(t)B(t)} = 4D,
$$
which implies
$$
\lim_{t \to T} |A(t)-B(t)| = 0, \quad \lim \limits_{t \to T}A(t) =  \lim \limits_{t \to T}B(t) = 2D.
$$
We have shown in particular that
$$
\lim \limits_{t \to T} ||v(t,\cdot)||_{\dot{H}^1(B^{(2)}_{\sigma(t)}(0))}^2 = 4 \pi D = E(Q).
$$
The part 
$$
 ||v(T,\cdot)||_{\dot{H}^1(B_{\sigma(t)}(0))} \to 0
$$
is negligible by absolute continuity of the Lebesgue integral and so is the part
$$
 ||v(t,\cdot) - v(T,\cdot)||_{\dot{H}^1(B_{\eta}(0) \setminus B_{\sigma(t)}(0))} \to 0
$$
by (\ref{eq: no concentration at self-similar scale}). The $L^2$-part causes no issue by dominated convergence. This finishes the proof.
\end{proof}

\section{Coercivity}

Let $Z \in C^{\infty}_c(\mathbb R^{2D+2})$ with $\langle Z, \Lambda W \rangle_{L^2(\mathbb R^{2D+2})} \neq 0$. We note that for $D \geq 2$, $2D + 2 \geq 5$, one has $\dot{H}^2_{rad}(\mathbb R^{2D+2}) \subset L^2_{loc}(\mathbb R^{2D+2})$. In the following, we will use the following abuse of notation:
$$
\langle Z \rangle^{\perp}_{\mathbb R^{2D+2}} = \{g \in \dot{H}^2_{rad}(\mathbb R^{2D+2}) : \langle g, Z \rangle_{L^2(\mathbb R^{2D+2})} = 0\}
$$
as the inner product is well-defined, since $Z$ is compactly supported.

The coercivity (for one bubble) states $||Hg||_{L^2} \simeq ||g||_{\dot{H}^2}$ for any $g \in \dot{H}^2_{rad}(\mathbb R^{2D+2}) \cap \langle Z \rangle^{\perp}_{\mathbb R^{2D+2}}$. Indeed, one has
$$
||Hg||_{L^2} \lesssim ||\Delta g||_{L^2} + ||r^{-2}g||_{L^2} \lesssim ||g||_{\dot{H}^2}
$$
using Rellich inequality if $D \geq 2$. 

The reverse inequality will be a consequence of the following. Observe that the map $H : \dot{H}^2_{rad}(\mathbb R^{2D+2}) \to L^2(\mathbb R^{2D+2})$ is injective on $\dot{H}^2_{rad} \cap \langle Z \rangle^{\perp}_{\mathbb R^{2D+2}}$. Otherwise, one would have $\Lambda W \in \dot{H}^2_{rad} \cap \langle Z \rangle^{\perp}_{\mathbb R^{2D+2}}$ meaning that $\langle \Lambda W, Z \rangle_{L^2(\mathbb R^{2D+2})}= 0$, which is a contradiction. 

From the injectivity of $H$ on $\dot{H}^2_{rad}(\mathbb R^{2D+2}) \cap \langle Z \rangle^{\perp}_{\mathbb R^{2D+2}}$, we can get another uniform (with respect to $\lambda$) coercivity result on $B_{\delta}(0)$ with
\begin{align*}
    H_{\lambda}  g &= -\Delta g - r^{-2} f'(Q_{\lambda})g.
\end{align*}
\begin{proposition}\label{prop:coercivity}
    There exists $\alpha_0 = \alpha_0(H, B_{\delta}(0), Z) > 0$ and constants $c_1 < c_2$, $c_i = c_i(H, B_{\delta}(0), Z)$, for which if $0 < \lambda \leq \alpha_0$, $g \in \dot{H}^2_{rad,0}(B_{\delta}(0)) \subset L^2(B_{\delta}(0))$ (where $\dot{H}^2_0$ denotes the closure of $C^{\infty}_c$ in the $\dot{H}^2$-norm) and $g \perp_{L^2(B_{\delta}(0))} Z_{\lambda}$, then
$$
c_1 ||g||_{\dot{H}^2(B_{\delta}(0))} \leq ||H_{\lambda} g||_{L^2(B_{\delta}(0))} \leq c_2 ||g||_{\dot{H}^2(B_{\delta}(0))}.
$$
\end{proposition}
\begin{proof}

The existence of $c_2$ follows from Rellich's inequality (as $f'(Q_{\lambda})$ is bounded uniformly in $\lambda$). Assume $c_1$ does not exist. Then there exist sequences $\alpha_n \to 0$, $\lambda_n \leq \alpha_n$, $g_n \in \dot{H}^2_{rad,0}(B_{\delta}(0))$, $g_n \perp_{L^2(B_{\delta}(0))} Z_{\lambda_n}$, $||g_n||_{\dot{H}^2_{rad}(B_{\delta}(0))} = 1$, $||H_{\lambda_n}g_n||_{L^2(B_{\delta}(0))} \to 0$.

Note that for any $A  > 1$, $\lambda > 0$,
$$
f'(Q) \leq \begin{cases}
    8D^2 r^2, \quad r \leq 1 \\
    8D^2 r^{-2}, \quad r \geq 1
\end{cases}, \quad 
f'(Q_{\lambda}) \leq \begin{cases}
    8D^2 A^{-2} , \quad r \leq A^{-1}\lambda \\
    8D^2, \quad r \in [A^{-1}\lambda,A\lambda] \\
    8D^2 A^{-2}, \quad r \geq A\lambda
\end{cases}
$$
Thanks to Rellich's inequality,
\begin{align*}
    1 = ||g_n||_{\dot{H}^2_{rad}(B_{\delta}(0))} &\lesssim ||\Delta g_n||_{L^2(B_{\delta}(0))} \leq ||H_{\lambda_n} g_n||_{L^2(B_{\delta}(0))}  + ||r^{-2}f'(Q_{\lambda_n})g_n||_{L^2(B_{\delta}(0))} \\
    &\lesssim o_n(1) + A^{-2}||r^{-2}g_n||_{L^2(B_{\delta}(0))}  + ||1_{ [A^{-1}\lambda_n,A\lambda_n]}r^{-2}g_n||_{L^2(B_{\delta}(0))} \\
    &\lesssim o_n(1) + A^{-2} +  ||1_{ [A^{-1}\lambda_n,A\lambda_n]}r^{-2}g_n||_{L^2(B_{\delta}(0))} \lesssim 1.
\end{align*}
Hence, by setting $n \geq N_0 \gg 1$ and $A \gg 1$, we deduce that
$$
||1_{ [A^{-1}\lambda_n,A\lambda_n]}r^{-2}g_n||_{L^2(B_{\delta}(0))}  \simeq 1.
$$
Let $h_n(x) = \lambda_n^{D-1}g_n(\lambda_nx)$. Then 
$$
||h_n||_{\dot{H}^2_{rad}(B_{\delta/\lambda_n})} = 1, \quad \langle h_n, Z \rangle_{L^2_{rad}(B_{\delta/\lambda_n})} = 0, \quad ||H h_n||_{L^2(B_{\delta/\lambda_n})} =  ||H_{\lambda_n} g_n ||_{L^2(B_{\delta})} \to 0,
$$
as well as 
$$
||1_{ [A^{-1},A]}r^{-2}h_n||_{L^2(B_{\delta/\lambda_n}(0))} \simeq 1
$$
Let $h \in \dot{H}^2_{rad}(\mathbb R^{2D+2})$ be the weak $\dot{H}^2(\mathbb R^{2D+2})$-limit of the sequence $(h_n)$. In particular, the convergence is strong in $L^2_{loc}(\mathbb R^{2D+2})$ by the Rellich–Kondrachov Theorem. Hence, $h \neq 0$ and observe that $\langle h, Z \rangle_{L^2(\mathbb R^{2D+2})} = 0$ (i.e., $h \in \langle Z \rangle^{\perp}_{\mathbb R^{2D+2}}$) as $Z$ has compact support. Moreover, $H h_n \to Hh$ weakly in $L^2(\mathbb R^{2D+2})$. Indeed, $\Delta h_n \to \Delta h$ weakly and
\begin{align*}
    |\langle |x|^{-2} f'(Q) (h_n-h), \phi \rangle| &\leq ||h_n-h||_{L^2(\supp \phi)} \cdot |||x|^{-2} \phi||_{L^2(\mathbb R^{2D+2})} \\
    &\leq ||h_n-h||_{L^2(\supp \phi)} \cdot ||\phi||_{\dot{H}^2(\mathbb R^{2D+2})} \to 0
\end{align*}
by Rellich inequality, for any $\phi \in C^{\infty}_c(\mathbb R^{2D+2})$. But for any test function,
\begin{align*}
    \langle Hh, \phi \rangle_{L^2(\mathbb R^{2D+2})} = \lim_{n \to +\infty} \langle H h_n, \phi \rangle \leq \lim_{n \to +\infty} ||H h_n||_{L^2(B_{\delta/\lambda_n}(0))} \cdot ||\phi||_{L^2(\mathbb R^{2D+2})} \to 0,
\end{align*}
which means that $H h = 0$. But on $\dot{H}^2_{rad}(\mathbb R^{2D+2}) \cap \langle Z \rangle^{\perp}_{\mathbb R^{2D+2}}$, $H$ is injective as previously seen. This contradicts $h \neq 0$.
    
\end{proof}

\section{Orbital Stability}\label{sec:orbital stability}
In what follows, we will work at the level of $u(r,t)$. 

Fix $\delta < 1$ (e.g. $\delta = 1/4$) a constant for which the profile decomposition applies. Fix $Z = ||\Lambda W||_{L^2(\mathbb R^{2D+2})}^{-1} (\Lambda W) \cdot \chi_{r \leq R_0}$ where $R_0 > 1$ is large enough so that $\langle Z, \Lambda W \rangle_{L^2(\mathbb R^{2D+2})} \geq 1/2$. 

The orbital stability is a classical result which follows from the implicit function theorem. For one bubble, the proof is as follows:
\begin{proposition}\label{prop: orbital stability}
For all $\delta_2 > 0$ small enough (depending on $\delta, \Lambda W, Z, D$), there exists $\delta_1 < \delta_2$  such that for all $\varepsilon \in \{-1,1\}$ and $0 < \lambda_0 \leq \delta/(2R_0)$, there is a unique $C^1$-function $\lambda: B_{\delta_1}(\varepsilon W_{\lambda_0}) \subset H^1(B_{\delta}(0)) \rightarrow \exp(B_{\delta_2}(y_0)) \subset (e^{-1}\lambda_0,e\lambda_0)$, $\lambda_0 = e^{y_0}$, for which $\lambda(\varepsilon W_{\lambda_0}) = \lambda_0$ and
\[
v := u- \varepsilon W_{\lambda(u)} \perp_{L^2(B_{\delta}(0))}  Z_{\lambda(u)}.
\]
In other words, if $||u-\varepsilon W_{\lambda_0}||_{H^1(B_{\delta}(0))}  \leq \delta_1$, then
\[
u = \varepsilon W_{\lambda(u)} +v, \quad v \perp_{L^2(B_{\delta}(0))} Z_{\lambda(u)}.
\]
The family of $C^1$-maps $(\lambda_{\lambda_0,\varepsilon})$ is compatible in the sense that
$$
\lambda_{\lambda_0,\varepsilon} = \lambda_{\lambda_1,\tilde{\varepsilon}}
$$
on overlaps
$$
B_{\delta_1}(\varepsilon W_{\lambda_0}) \cap B_{\delta_1}(\tilde{\varepsilon} W_{\lambda_1}).
$$
\end{proposition}



\begin{proof} 

Fix $\varepsilon \in \{-1,1\}$. Let $\Phi : H^1(B_{\delta}(0)) \times (0,+\infty) \rightarrow \mathbb R$ be the $C^1$ map defined via
\[
(u,\lambda) \mapsto \langle u - \varepsilon W_{\lambda}, \varepsilon \lambda^{-2} Z_{\lambda} \rangle_{L^2(B_{\delta}(0))}.
\]
For any $\lambda_0 \in \mathbb R_{> 0}$ fixed,
$$
\Phi(\varepsilon W_{\lambda_0}, \lambda_0) = 0.
$$
Fix $\chi \in C^{\infty}_c(B_{\delta}(0))$, $0 \leq \chi \leq 1$, such that $\chi = 1$ on $B_{\delta/2}$. If $\lambda^{-1} \delta \geq 2R_0$, then
\begin{align*}
    \lambda \partial_{\lambda} \Phi(u,\lambda) &= \langle  (\Lambda W)_{\lambda}, \lambda^{-2}  Z_{\lambda} \rangle_{L^2(B_{\delta}(0))}  - \langle u-\varepsilon W_{\lambda}, \varepsilon  \lambda^{-2}[(\Lambda Z)_{\lambda} + 2Z_{\lambda}] \rangle_{L^2(B_{\delta}(0))} \\
    &\geq \frac{1}{2}- \langle u-\varepsilon W_{\lambda}, \varepsilon  \lambda^{-2}[(\Lambda Z)_{\lambda} + 2Z_{\lambda}] \rangle_{L^2(B_{\lambda R_0}(0))} \\
    &\geq \frac{1}{2}- \langle (u-\varepsilon W_{\lambda})\chi, \varepsilon  \lambda^{-2}[(\Lambda Z)_{\lambda} + 2Z_{\lambda}] \rangle_{L^2(\mathbb R^{2D+2})} \\
    &\geq \frac{1}{2} - \delta^{-1}C( Z, D) ||u-\varepsilon W_{\lambda}||_{H^1(B_{\delta}(0))}.
\end{align*}
In other words, for $\lambda = e^y$ with $y \leq \log(\delta/(2R_0))$,
\begin{align*}
   \partial_y \left[ \Phi(u,e^y) \right] &= \langle  (\Lambda W)_{\lambda}, \lambda^{-2}  Z_{\lambda} \rangle_{L^2(B_{\delta}(0))}  - \langle u-\varepsilon W_{\lambda}, \varepsilon  \lambda^{-2}[(\Lambda Z)_{\lambda} + 2Z_{\lambda}] \rangle_{L^2(B_{\delta}(0))} \\
    &\geq \frac{1}{2} - \delta^{-1}C( Z, D) ||u-\varepsilon W_{\lambda}||_{H^1(B_{\delta}(0))}.
\end{align*}
 Similarly, for $\delta/(2R_0) \geq \lambda' \geq \lambda$,
\begin{align*}
      \partial_y \left[ \Phi(u,e^y) \right] - \partial_y \left[ \Phi(u',e^{y'}) \right]  &= \langle  (\Lambda W),  Z \rangle_{L^2(B_{\delta/\lambda}(0) \setminus B_{\delta/\lambda'}(0))}  \\
     &- \langle u-\varepsilon W_{\lambda}, \varepsilon  \lambda^{-2}[(\Lambda Z)_{\lambda} + 2Z_{\lambda}] \rangle_{L^2(B_{\delta}(0))} \\
     &+ \langle u'-\varepsilon W_{\lambda'}, \varepsilon  \lambda'^{-2}[(\Lambda Z)_{\lambda'} + 2Z_{\lambda'}] \rangle_{L^2(B_{\delta}(0))} \\
    &\leq   \delta^{-1} C( Z, D) (||u-\varepsilon W_{\lambda}||_{H^1(B_{\delta}(0))}  +  ||u'-\varepsilon W_{\lambda'}||_{H^1(B_{\delta}(0))}),
\end{align*}
where we used that
$$
\langle  (\Lambda W),  Z \rangle_{L^2(B_{\delta/\lambda}(0) \setminus B_{\delta/\lambda'}(0))}  = 0,
$$
since $Z$ is supported on $B_{R_0}(0)$. Moreover, $\Phi$ is $C^1$ in $u$, with Lipschitz bound 
$$
|\Phi(u,\lambda) - \Phi(u',\lambda)| = | \langle u - u', \varepsilon \lambda^{-2} Z_{\lambda} \rangle_{L^2(B_{\delta}(0))}| \leq \delta^{-1} C(Z, D) ||u-u'||_{H^1(B_{\delta}(0))}.
$$
Note also that
$$
||(\Lambda W)_{\theta}||_{H^1(B_{\delta}(0))} \leq (1+1_{D = 1}|\log(\delta)|) C(\Lambda W, D), \quad \theta \in (0,e).
$$
Hence, there exists a constant $C(\Lambda W, Z, D) > 1$ such that if 
\begin{align}
    \delta_2 &= \delta_2(\delta,\Lambda W, Z, D) \leq \frac{\delta}{12 (1 + 1_{D = 1} |\log(\delta)|) C(\Lambda W, Z, D)^2}  \\
\delta_1 &= \delta_1(\delta,\Lambda W, Z, D) \leq \frac{\delta_2}{6\delta^{-1}C(\Lambda W,Z,D)}  \label{eq: delta_1 definition}
\end{align}
and $\lambda_0 \leq \delta/(2R_0) \leq 1/2$ are fixed, the restriction of $F(u,y) = \Phi(u,e^y)$ to $B_{\delta_1}(\varepsilon W_{\lambda_0}) \times B_{\delta_2}(y_0) \subset H^1(B_{\delta}(0)) \times \mathbb R$, where $\lambda_0 = e^{y_0}$, satisfies
\begin{align}
F(\varepsilon W_{\lambda_0},y_0) &= 0, \quad L_0 = \partial_{y}F(\varepsilon W_{\lambda_0},y_0) \geq 1/2 \notag \\
    |F(u,y_0)-0| &\leq \delta^{-1}C(\Lambda W,Z,D)||u-\varepsilon W_{\lambda_0}||_{H^1(B_{\delta}(0))} \notag \\
    &\leq \delta^{-1} C(\Lambda W,Z,D)\delta_1  \leq \delta_2/6 \leq \delta_2 L_0/3 \\
        ||W_{e^y}-W_{e^{y_0}}||_{H^1(B_{\delta}(0))} &\leq |e^y-e^{y_0}| \cdot ||(\Lambda W)_{\theta}||_{H^1(B_{\delta}(0))}, \quad \theta \in [e^y,e^{y_0}] \subset (0,e) \notag \\
        &\leq |y-y_0|  \cdot (1 + 1_{D=1}|\log(\delta)|) C(\Lambda W, Z, D) \notag \\
            &\leq \delta_2 \cdot  (1 + 1_{D=1}|\log(\delta)|) C(\Lambda W, Z, D) \label{eq: estimate difference of W} \\
    |\partial_{y} F(u,y) - L_0| &\leq \delta^{-1}C(\Lambda W,Z,D)(\delta_1 + ||W_{e^y}-W_{e^{y_0}}||_{H^1(B_{\delta}(0))}) \notag \\ 
    &\leq \delta^{-1}(1 + 1_{D=1}|\log(\delta)|)C(\Lambda W, Z,D)^2(\delta_1 + \delta_2) \notag \\
    &\leq 1/6 \leq L_0/3 \label{eq: estimate on partial_y F}.
\end{align}

The Implicit Function Theorem (see \cite{jendrej:hal-03365769}, Remark 2.25) yields the existence of a unique $C^1$-function $\lambda: B_{\delta_1}(\varepsilon W_{\lambda_0}) \subset H^1(B_{\delta}(0)) \rightarrow \exp\left( B_{\delta_2}(y_0) \right) \subset (0,+\infty)$ for which $\lambda(\varepsilon W_{\lambda_0}) = \lambda_0$ and $\Phi(u,\lambda(u)) = 0$, i.e.,
\[
v := u- \varepsilon W_{\lambda(u)} \perp_{L^2(B_{\delta}(0))}  Z_{\lambda(u)}.
\]
In other words, if $||u-\varepsilon W_{\lambda_0}||_{H^1(B_1(0))}  \leq \delta_1$, then
\[
u = \varepsilon W_{\lambda(u)} +v, \quad v \perp_{L^2(B_{\delta}(0))} Z_{\lambda(u)}.
\]
with $e^{-1}\lambda_0 \leq e^{y_0-\delta_2} \leq \lambda(u) \leq e^{y_0 + \delta_2} \leq e\lambda_0$.

We now shrink $\delta_1,\delta_2$, if necessary, so that the maps $\lambda$ given by the implicit function theorem are compatible on overlapping balls. From (\ref{eq: estimate on partial_y F}), 
$$
\partial_yF_\varepsilon(u,y)\ge 2L_0/3 \geq 1/3
$$
for all $(u,y) \in B_{\delta_1}(\varepsilon W_{e^{y_0}}) \times B_{\delta_2}(y_0)$. Define
$$
m(\eta):=
\inf_{\substack{0<\lambda,\mu\le \delta/(2R_0)\\
|\log(\lambda/\mu)|\ge \eta}}
\|W_\lambda-W_\mu\|_{H^1(B_\delta)}>0, \quad \eta > 0,
$$
and
$$
m_\pm:=
\inf_{0<\lambda,\mu\le \delta/(2R_0)}
\|W_\lambda+W_\mu\|_{H^1(B_\delta)}>0.
$$
Choose $0 < \tilde{\delta}_2 < \delta_2$ smaller so that
$$
(1 + 1_{D=1}|\log(\delta)|) C(\Lambda W, Z, D) \cdot \tilde{\delta}_2\le \delta_1/2,
$$
where $C$ is the constant in the upper bound
$$
\|W_{e^y}-W_{e^{y_0}}\|_{H^1(B_{\delta}(0))}\leq |y-y_0| \cdot (1 + 1_{D=1}|\log(\delta)|) C(\Lambda W, Z, D)
$$
from (\ref{eq: estimate difference of W}).
Choose $0 < \tilde{\delta}_1 < \delta_1/2$ satisfying (\ref{eq: delta_1 definition}) with $\tilde{\delta}_2$ instead of $\delta_2$, as well as
$$
2\tilde{\delta}_1 < m(\tilde{\delta}_2/2),
\qquad
2\tilde{\delta}_1 < m_\pm.
$$
Now suppose
$$
u\in B_{\tilde{\delta}_1}(\varepsilon W_{\lambda_0})
   \cap B_{\tilde{\delta}_1}(\widetilde\varepsilon W_{\lambda_1}),
\qquad
\lambda_j=e^{y_j}.
$$
If \(\widetilde\varepsilon=-\varepsilon\), then
$$
\|W_{\lambda_0}+W_{\lambda_1}\|_{H^1(B_\delta)}
\le
\|u-\varepsilon W_{\lambda_0}\|_{H^1}
+
\|u+\varepsilon W_{\lambda_1}\|_{H^1}
<2\tilde{\delta}_1,
$$
contradicting \(2\tilde{\delta}_1<m_\pm\). Hence an overlap can occur only when
\(\widetilde\varepsilon=\varepsilon\).

In that case,
$$
\|W_{\lambda_0}-W_{\lambda_1}\|_{H^1(B_\delta)}
\le 2\tilde{\delta}_1,
$$
so by the definition of $m(\tilde{\delta}_2/2)$, we get
$$
|y_0-y_1|<\tilde{\delta}_2/2.
$$
Therefore the two intervals
$$
B_{\tilde{\delta}_2}(y_j) =(y_j-\tilde{\delta}_2,y_j+\tilde{\delta}_2),\qquad j=0,1,
$$
have a connected union. Moreover, for every $y\in B_{\tilde{\delta}_2}(y_0)\cup B_{\tilde{\delta}_2}(y_1)$, choosing \(j\)
such that \(y\in B_{\tilde{\delta}_2}(y_j)\), we have
$$
\|u-\varepsilon W_{e^y}\|_{H^1(B_\delta)}
\le
\|u-\varepsilon W_{\lambda_j}\|_{H^1(B_\delta)}
+
\|W_{\lambda_j}-W_{e^y}\|_{H^1(B_\delta)}
\le
\delta_1/2 + \delta_1/2.
$$
Hence
$$
\partial_yF_\varepsilon(u,y)\ge \frac{1}{3}
$$
for 
$$
u \in  B_{\tilde{\delta}_1}(\varepsilon W_{\lambda_0})
   \cap B_{\tilde{\delta}_1}(\widetilde\varepsilon W_{\lambda_1}), \quad y \in B_{\tilde{\delta}_2}(y_0)\cup B_{\tilde{\delta}_2}(y_1), \quad \lambda_i = e^{y_i}.$$
   
Thus, \(F(u,\cdot)\) is strictly increasing on $B_{\tilde{\delta}_2}(y_0)\cup B_{\tilde{\delta}_2}(y_1)$, and so it has at most one zero there. But applying the implicit function theorem on both balls yields two $C^1$ functions $\lambda_0, \lambda_1$ for which
$$
F(u, \log \lambda_i(u)) = 0,
$$
meaning that $\lambda_0(u)=\lambda_1(u)$.
\end{proof}

Here is how we apply the result. Assume $D \geq 2$. Fix $\alpha_0 > 0$, $c_1 < 1 < c_2$ as in the Coercivity Result (Proposition \ref{prop:coercivity}). 

Let $C_S = C_S(D,\delta) > 1$ be an upper bound on the operator norms of the Sobolev embeddings 
\begin{equation}
    \dot{H}^1_0(B_{\delta}(0) \subset \mathbb R^{n}) \hookrightarrow L^{\frac{2n}{n-2}}(B_{\delta}(0)), \quad \dot{H}^2_0(B_{\delta}(0)) \hookrightarrow L^{\frac{2n}{n-4}}(B_{\delta}(0))\label{eq: sobolev embedding}
\end{equation}
with $n = 2D + 2 > 4$ (where $\dot{H}_0^i$ denotes the closure of smooth, compactly supported functions in the $\dot{H}^i$ norm).  Let also $C_R(D,\delta) > 1$ be a constant for which Hardy-Rellich's inequalities hold on $\dot{H}^1_0(B_{\delta}(0))$ and $\dot{H}^2_0(B_{\delta}(0))$, i.e.
\begin{equation}
    ||r^{-1} h||_{L^2(B_{\delta}(0))} \leq C_R ||h||_{\dot{H}^1(B_{\delta}(0))}, \quad ||r^{-2} h||_{L^2(B_{\delta}(0))} \leq C_R ||h||_{\dot{H}^2(B_{\delta}(0))}. \label{eq: rellich inequ}
\end{equation}
In particular, observe that
\begin{align}
    |\nabla(h \chi_{r_0/2})| \leq |\nabla h| + |h \nabla \chi_{r_0/2}| 
    &\leq |\nabla h| + 2 B_{\chi}^{(1)} |r^{-1} h|,  \label{eq: upper bound on nabla h cutoff}
\end{align}
which implies
\begin{align}
    || h \chi_{r_0/2}||_{H^1(B_{\delta}(0))}  &\leq || h ||_{H^1(B_{\delta}(0))}+  || h \nabla \chi_{r_0/2}||_{L^2(B_{\delta}(0))} \notag \\
    &\leq || h ||_{H^1(B_{\delta}(0))}+  2 B_{\chi}^{(1)}|| r^{-1} h||_{L^2(B_{\delta}(0))} \notag \\
    &\leq (1 + 2 C_R B_{\chi}^{(1)}) || h ||_{H^1(B_{\delta}(0))} \label{eq: rellich ineq with cutoff}
\end{align}
if $r_0 < \delta$ for any $h \in \dot{H}^1_0(B_{\delta}(0))$.
 Finally, let $\pi M \in \pi \mathbb N_{> 0}$ be a large multiple of $\pi$ for which
\begin{equation}
    |r^D u_*| + |r^Du(t,r)| + |r^DW| \leq M\pi. \label{eq: large multiple of pi}
\end{equation}
Such a $M$ exists by comparison principle for $v(r,t)$ (see \cite{samuelian2025blowuptrees}), as the initial and boundary data $v_0$ is bounded by a large multiple of $\pi$ on $B^2 \times [0,T]$.

 Let $\delta_2 \ll_{\Lambda W, Z, D, \delta} 1$ be small enough in the Orbital Stability Result (Proposition \ref{prop: orbital stability}). The appropriate smallness is chosen so that (\ref{eq: smallness of inner product with g_n}) will hold. Fix $\delta_1 < \delta_2$ as in Proposition \ref{prop: orbital stability}. Next, assume that $r_0 < \delta/2$ is small enough (depending on $\delta_1, \delta_2$), so that for all $\tilde{r}_0 \leq r_0$,
$$
 ||u_* \chi_{\tilde{r}_0}||_{H^1(B_{\delta}(0))}  \leq \delta_1/2
 $$
and take $r_0$ even smaller for (\ref{eq: smallness of inner product with g_n}) to hold. It is possible to take such a $r_0$ because (\ref{eq: upper bound on nabla h cutoff}) implies, by Hardy-Rellich and dominated convergence, that
$$
 \lim \limits_{r_0 \to 0} ||u_* \chi_{r_0}||_{H^1(B_{\delta}(0))} =  \lim \limits_{r_0 \to 0} ||(u_* \chi_{\delta/2}) \chi_{r_0}||_{H^1(B_{\delta}(0))}   = 0.
$$
Next, fix $\eta < \delta_1/2$ small enough so that (\ref{eq: smallness of inner product with g_n}) holds. Assume that there is a sequence $(\lambda_0^{(n)}) \to 0$ and a sequence of time-intervals $[t_0^{(n)},t_1^{(n)}]$ on which
\begin{equation}
     \sup_{t \in [t_0^{(n)},t_1^{(n)}]} ||u(\cdot,t) - u_*(\cdot) -  W_{\lambda_0^{(n)}}||_{H^1(B_{\delta}(0))} \leq \eta < \delta_1/2. \label{eq: profile decomposition for u on sequence of intervals}
\end{equation}
For all $n$ large enough, one can apply orbital stability. We obtain a scale $\lambda_n(t) \in C^1([t_0^{(n)},t_1^{(n)}],(e^{-1}\lambda_0^{(n)},e\lambda_0^{(n)}))$ (technically, it should be $\lambda_n(u(t))$) and a remainder $g_n(t,r)$ for which
$$
u(t,r) - (1-\chi_{r_0})u_*(r)=  W_{\lambda_n(t)}  + g_n(t,r), \quad t \in [t_0^{(n)},t_1^{(n)}], 0 \leq r \leq \delta 
$$
and $g_n$ satisfies the orthogonality condition
$$
0 = \langle g_n(t,\cdot), Z_{\lambda_n(t)} \rangle_{L^2(B_{\delta}(0))} = \langle g_n(t,\cdot), Z_{\lambda_n(t)} \rangle_{L^2(B_{\lambda(t)R_0}(0))}, \quad \quad t \in [t_0^{(n)},t_1^{(n)}].
$$
Replacing $g_n$ by $\tilde{g}_n = g_n + (1-\chi_{r_0})u_*(r)$, the same orthogonality condition is satisfied if $r_0/(4R_0) \geq (\lambda_0^{(n)})$. 

Moreover, $r^D\tilde{g}_n$ are bounded by a multiple of $\pi$ thanks to (\ref{eq: large multiple of pi}) and
\begin{align*}
    ||\tilde{g}_n \chi_{r_0/2}||_{H^1(B_{\delta}(0))} &= ||g_n \chi_{r_0/2}||_{H^1(B_{\delta}(0))} \leq ||u_*(r) \chi_{r_0/2}||_{H^1(B_{\delta}(0))} \\
    &+ ||[ u(\cdot,t) - u_*(r) -  W_{\lambda_0^{(n)}}] \chi_{\delta/2} \chi_{r_0/2}||_{H^1(B_{\delta}(0))} \\
    &+ ||[W_{\lambda_n(t)} - W_{\lambda_0^{(n)}}]\chi_{\delta/2} \chi_{r_0/2}||_{H^1(B_{\delta}(0))} \\
    &\leq ||u_*(r) \chi_{r_0/2}||_{H^1(B_{1}(0))} \\
    &+(1 + 2 C_R B_{\chi}^{(1)})  ||[ u(\cdot,t) - u_*(r) -  W_{\lambda_0^{(n)}}] \chi_{\delta/2} ||_{H^1(B_{\delta}(0))}\\
    &+ (1 + 2 C_R B_{\chi}^{(1)})||[W_{\lambda_n(t)} - W_{\lambda_0^{(n)}}]\chi_{\delta/2}||_{H^1(B_{\delta}(0))}, 
 \end{align*}
 whenever $t \in [t_0^{(n)},t_1^{(n)}], D \geq 1$, where we used (\ref{eq: rellich ineq with cutoff}). Then, using the upper bound $|\nabla \chi_{\delta/2}| \leq 2 \delta^{-1} B_{\chi}^{(1)}$, we obtain 
\begin{align}
   ||\tilde{g}_n \chi_{r_0/2}||_{H^1(B_{\delta}(0))} &= ||g_n \chi_{r_0/2}||_{H^1(B_{\delta}(0))} \leq ||u_*(r) \chi_{r_0/2}||_{H^1(B_{1}(0))} \notag \\
    &+(1 + 2 C_R B_{\chi}^{(1)})(1+2\delta^{-1}  B_{\chi}^{(1)}) \eta \notag\\
    &+  C(\Lambda W,Z,D)(1 + 2 C_R B_{\chi}^{(1)})(1+2\delta^{-1}  B_{\chi}^{(1)}) \delta_2 \notag \\
    &= o_{r_0 \to 0}(1) + \mathcal{O}(\eta) + \mathcal{O}(\delta_2), \quad t \in [t_0^{(n)},t_1^{(n)}], \quad D \geq 1, \label{eq: triple o upper bound}
 \end{align}
 where we used (\ref{eq: estimate difference of W}), (\ref{eq: profile decomposition for u on sequence of intervals}) and 
 $$
 \lim \limits_{r_0 \to 0} ||u_* \chi_{r_0}||_{H^1(B_{\delta}(0))}  = 0.
$$
 If $\delta_2$ is chosen small enough, $\delta_1$ is chosen as in Proposition \ref{prop: orbital stability}, $r_0$ and $\eta$ are chosen small enough (depending on $\delta_2,\delta_1$), then we obtain that for $t \in [t_0^{(n)},t_1^{(n)}]$,
 \begin{align}
    D^2 (M\pi)^{2-p}C_S^{p} ||\tilde{g}_n \chi_{r_0/2}||_{H^1(B_{\delta}(0))}^{p-1} &\leq c_1/2,  \notag \\
     ||\tilde{g}_n \chi_{r_0/2}||_{H^1(B_{\delta}(0))} &\leq 1, \quad p = 1 + 2/D, \notag \\
    ||\tilde{g}_n \chi_{r_0/2}||_{H^1(B_{\delta}(0))} &\leq  \tilde{C}^{-1} \min\{1,||\Lambda W||_{\dot{H}^1(B_{\delta}^{(2D+2)}(0))}^2,||Q||_{\dot{H}^1(B_1^{(2)}(0))}\} \label{eq: smallness of inner product with g_n}
\end{align}
uniformly in $n$, where $C_S$ comes from the Sobolev embedding (\ref{eq: sobolev embedding}),
$$
\tilde{C} = 4\left(\max\{A_{(\Lambda + 2)\Lambda W}, A_Z\}C_R ||r \langle r \rangle^{-2D}||_{L^2(\mathbb R^{2D+2})} + C(D,\delta) + 1 \right)
$$
is a constant for which $C_R$ comes from Hardy-Rellich inequality, $C(D,\delta)$ comes from Lemma \ref{lemma:equivalence of remainder convergence} and $A$ is obtained from inequality (\ref{eq: estimate on g chi lambda^-2 psi}) applied with $\psi = (\Lambda + 2)\Lambda W$ and $\psi = Z$. 
 


\section{$L^2$-estimate for $g(t,r)$}
Assume that $g$ is smooth, $D \geq 2$, $r_0 < \delta/2 < 1$, $\lambda(t) \leq \min\{\alpha_0,r_0/(4R_0)\}$,
\begin{align*}
    g(t,r) = u(t,r) - W_{\lambda(t)}, \quad t &\in [t_0,t_1], r \in B_{\delta}(0) \\
    |r^Dg(t,r)| \leq M \pi, \quad t &\in [t_0,t_1], r \in B_{\delta}(0) \\
    \langle g(t,\cdot), Z_{\lambda(t)} \rangle_{L^2(B_{\delta}(0))} = 0, \quad t &\in [t_0,t_1] \\
    D^2 (M\pi)^{2-p} C_S^{p} ||g \chi_{r_0/2}||^{p-1}_{\dot{H}^1(B_{\delta}(0))} \leq c_1/2. \quad t &\in [t_0,t_1], \ p = 1+2/D,
\end{align*}
where $R_0$ comes from $Z$, $C_S = C_S(D,\delta)$ is the Sobolev embedding constant from (\ref{eq: sobolev embedding}) and $\alpha_0(H,B_{\delta}(0),Z)$ and $c_1 = c_1(H,B_{\delta}(0),Z)$ are constants from the coercivity result (Proposition \ref{prop:coercivity}). Those hypotheses are satisfied for $g = \tilde{g}_n$ in Section \ref{sec:orbital stability}. The goal is to prove an $L^2$-estimate for the remainder $g(t,r)$ of the form
\begin{equation}
    \int_{t_0}^{t_1} ||g(t)||_{\dot{H}^2(B_{\delta}(0))}^2 dt \leq C(\delta,r_0,u,D,\chi,M,H,Z) <  +\infty,  \label{eq: l^2 estimate for g}
\end{equation}
which is independent of $t_0,t_1,\lambda(t)$.

First, observe that
\begin{equation}
    (1-\chi_{r_0/4})\Delta g = (1-\chi_{r_0/4})\Delta u -  (1-\chi_{r_0/4})\Delta W_{\lambda(t)} \label{eq: upper bound of g(t) in terms of Delta u}
\end{equation}
has bounded $L^2(B_{\delta}(0))$-norm uniformly in $t$, as $u$ is smooth away from the origin and $|\Delta W_{\lambda(t)}| \lesssim r^{-D-2} \in L^2(\mathbb R^{2D+2} \setminus B_{r_0/2}(0))$.

For the part localized at the origin, the idea is to prove the pointwise bound
\begin{equation}
     ||g(t)\chi_{r_0/4}||_{\dot{H}^2(B_{\delta}(0))}\leq ||u_t\chi_{r_0/4}||_{L^2(B_{\delta}(0))}  + C(\delta,r_0,D,\chi,M,H,Z), \label{eq: upper bound of g(t) in terms of u_t}
\end{equation}
which is integrable by (\ref{eq: finite L^2-norm of v_t}). For this, we look at the PDE for $g(t,r)$, i.e.,
\begin{equation}
    u_t = - H_{\lambda(t)} g + NL(g),\label{eq: pde for g(t,r)}
\end{equation}
where 
\begin{align*}
    H_{\lambda(t)}  g &= -\Delta g - r^{-2} f'(Q_{\lambda(t)})g.
\end{align*}
We multiply by $\chi_{r_0/4}$, so that 
\begin{equation}
    u_t \chi_{r_0/4} = - H_{\lambda(t)} (g\chi_{r_0/4}) - [\Delta,\chi_{r_0/4}]g + NL(g)\chi_{r_0/4},\label{eq: pde for g(t,r)chi}
\end{equation}
where
$$
    ||H_{\lambda(t)}(g \chi_{r_0/4})||_{L^2(B_{\delta}(0))} \geq c_1(H,\delta,Z) ||g \chi_{r_0/4}||_{\dot{H}^2(B_{\delta}(0))} 
$$
by coercivity and
\begin{align*}
    NL(g) &= r^{-D-2}[f(r^D W_{\lambda} + r^Dg) - f(r^DW_{\lambda}) - f'(r^DW_{\lambda})r^Dg] \\
    |NL(g)| &\leq D^2 r^{-D-2}|r^Dg|^2, \quad |f''(u)| \leq 2D^2 \\
     ||NL(g) \chi_{r_0/4}||_{L^2(B_{\delta}(0))} &= ||NL(g) \chi_{r_0/4}\chi_{r_0/2}||_{L^2(B_{\delta}(0))} \\
     &\leq D^2 ||(r^Dg)^{2-p} g^p \chi_{r_0/4}\chi_{r_0/2}||_{L^2(B_{\delta}(0))}, \quad p =  1 + 2/D\\
     &\leq D^2 (M\pi)^{2-p} || g^p \chi_{r_0/4}\chi_{r_0/2}||_{L^2(B_{\delta}(0))} \\
     &\leq D^2 (M\pi)^{2-p}C_S^{p} ||g \chi_{r_0/2}||^{p-1}_{\dot{H}^1(B_{\delta}(0))} ||g \chi_{r_0/4}||_{\dot{H}^2(B_{\delta}(0))}  \\
     &\leq c_1||g\chi_{r_0/4}||_{\dot{H}^2(B_{\delta}(0))}/2
\end{align*}
by Hölder's inequality and the Sobolev embeddings (\ref{eq: sobolev embedding}). Finally,
\begin{align*}
    |[\Delta,\chi_{r_0/4}]g| &= \left| \left( \chi_{r_0/4}''(r) + \frac{2D+1}{r} \chi_{r_0/4}'(r) \right) g + 2 \chi_{r_0/4}'(r)g_r \right| \\
    &\leq C(\chi,r_0,D) \left( \frac{|g \chi_{r_0/2}|}{r} +  |g_r| \right), \quad r \leq r_0/2 \\
    ||[\Delta,\chi_{r_0/4}]g||_{L^2(B_{\delta}(0))} &\leq C(\chi,r_0,D)C_R(D,\delta) ||g \chi_{r_0/2}||_{\dot{H}^1(B_{\delta}(0))} \\
    &\leq C(\delta,r_0,D,\chi,M,H,Z)
\end{align*}
using Rellich's inequality (\ref{eq: rellich inequ}) and our hypotheses on $g$.

\section{Modulation Estimates}
Assume that $g$ satisfies the hypotheses of the previous section on $[t_0,t_1]$. Further assume that $\lambda < 1$, $r_0/(4\lambda) > R_0$, 
\begin{equation}
    \left|\mathbb S^{2D+1} \right|(2D-2)^{-1}4D^2 \left( \frac{r_0}{4\lambda} \right)^{-(2D-2)} \leq ||\Lambda W||_{L^2(\mathbb R^{2D+2})}^2/4. \label{eq: definition of R_0}
\end{equation}
and
\begin{equation}
    ||g \chi_{r_0/2}||_{\dot{H}^1(B_{\delta}(0))} \leq \left( C_R \max_{\psi \in \mathcal{F} } ||r \psi||_{L^2(\mathbb R^{2D+2})}+1 \right)^{-1} \min\{1,||\Lambda W||_{L^2(B_{\delta}(0))}^2\}/4,  \label{eq: smallness of inner product for g not depending on n}
\end{equation}
where $\mathcal{F} = \{(\Lambda +2)\Lambda W, Z, 1_{D \geq 3} \Lambda W\}$ and $C_R$ is as in Lemma \ref{lemma: inner product inequalities with psi}.  Those hypotheses are satisfied for $g = \tilde{g}_n$ in Section \ref{sec:orbital stability}. Then one computes that:
\begin{theorem}\label{thm:modulation estimates}
Under the previous assumptions, if $t \in [t_0,t_1]$, 
    \begin{align}
    \left| \frac{\lambda_t}{\lambda} \right| &\leq C(\delta,r_0,D,\chi,M,H,Z) \left( \frac{|| g \chi_{r_0/2}||_{\dot{H}^2(B_{\delta}(0))} }{\lambda} + || g \chi_{r_0/2}||_{\dot{H}^2(B_{\delta}(0))}^2 \right), \label{eq: modulation, weaker form} \\
   \left| \frac{\lambda_t}{\lambda} +  \omega'(t) \right| &\leq C(\delta,r_0,D,\chi,M,H,Z) \cdot \left( 1 + ||g\chi_{r_0/2}||_{\dot{H}^2(B_{\delta}(0))}^2 \right), \label{eq: modulation, more precise}
\end{align}
where $\omega(t) \in C^1([t_0,t_1])$ is bounded by
$$
|\omega(t)| \leq C(\delta,r_0,D,\chi,M,H,Z)(1 + \sqrt{|\log(\lambda)|} \cdot \delta_{D = 2}).
$$
\end{theorem}

The proof follows the strategy of \cite{kim2025classificationglobaldynamicsenergycritical}, Section 5. The weaker modulation estimate for $D = 2$ is new and will allow to extend the nonexistence result to $D \geq 2$.
\begin{proof}
Observe that
$$
\partial_t \langle g , \lambda^{-2} \psi_{\lambda} \rangle = \langle g_t, \lambda^{-2} \psi_{\lambda} \rangle - \frac{\lambda_t}{\lambda} \langle g, \lambda^{-2} (\Lambda + 2) \psi_{\lambda} \rangle, \quad \partial_t W_{\lambda} = -\frac{\lambda_t}{\lambda} (\Lambda W)_{\lambda}
$$
for integrable $\psi$. We go back to the equation for $g\chi_{r_0/4}$ :
$$
\chi_{r_0/4}  \partial_t  W_{\lambda(t)} + g_t\chi_{r_0/4} = - H_{\lambda(t)} (g\chi_{r_0/4}) - [\Delta,\chi_{r_0/4}]g + NL(g)\chi_{r_0/4}
$$
and apply $\langle \cdot, \psi_{\lambda}\rangle$. One gets 
\begin{align}
    \frac{\lambda_t}{\lambda} \langle g \chi_{r_0/4}, \lambda^{-2} (\Lambda + 2) \psi_{\lambda} \rangle &-  \frac{\lambda_t}{\lambda} \langle   \chi_{r_0/4}(\Lambda W)_{\lambda}, \lambda^{-2} \psi_{\lambda} \rangle = - \partial_t \langle g \chi_{r_0/4}, \lambda^{-2} \psi_{\lambda} \rangle \notag \\
    &+ \langle -H_{\lambda}(g\chi_{r_0/4})   - [\Delta,\chi_{r_0/4}]g + NL(g)\chi_{r_0/4} , \lambda^{-2} \psi_{\lambda} \rangle   \label{eq: modulation estimate, general case}
\end{align}
 with the special case,
\begin{align}
     \frac{\lambda_t}{\lambda}\large(  \langle g\chi_{r_0/4}, \lambda^{-2} (\Lambda + 2) (\Lambda W)_{\lambda} \rangle  &-   \langle \chi_{r_0/4} (\Lambda W)_{\lambda}, \lambda^{-2} (\Lambda W)_{\lambda} \rangle \large) = - \partial_t \langle g\chi_{r_0/4}, \lambda^{-2} (\Lambda W)_{\lambda} \rangle  \notag
     \\
     &+ \langle - [\Delta,\chi_{r_0/4}]g + NL(g)\chi_{r_0/4} , \lambda^{-2} (\Lambda W)_{\lambda} \rangle.  \label{eq: modulation estimate, crucial case}
\end{align}
\begin{lemma} \label{lemma: inner product inequalities with psi}
 Let $|\psi| \leq A_{\psi} \langle y \rangle^{-2D}$ (e.g. $\psi \in \{Z,\Lambda W, (\Lambda + 2)\Lambda W, (\Lambda + 2)^2\Lambda W\}$) and similarly for $\tilde{\psi}$. Let $B_{\chi}^{(n)}$ be as in (\ref{eq: upper bound cut-off}), $c_2 = c_2(H,B_{\delta},Z)$ be as in the coercivity result (Proposition \ref{prop:coercivity}), 
 $$
I(D) = |S^{2D+1}| \int_0^{+\infty} \langle s \rangle^{-4D} s^{2D+1}ds + \sup_{r > 0} \frac{r^D}{(1+r)^{2D}}< +\infty.
 $$
 Let also $C_R(D,\delta) > 1$ be a constant for which Rellich's inequality holds on $\dot{H}^1_0(B_{\delta}(0))$ and $\dot{H}^2_0(B_{\delta}(0))$.
 
Assume $r_0 \leq \delta/2 < 1$ and $g \chi_{r_0/2} \in H^2(B_{\delta}(0))$. Then:
\begin{align*}
    \langle H_{\lambda}(g\chi_{r_0/4}), (\Lambda W)_{\lambda} \rangle &= 0, \quad |\langle \chi_{r_0/4} \psi_{\lambda}, \tilde{\psi}_{\lambda} \rangle| \leq \lambda^2 I(D)A_{\psi} A_{\tilde{\psi}} \\
    |\langle H_{\lambda}(g\chi_{r_0/4}), \lambda^{-2} \psi_{\lambda} \rangle |  &\leq  ||H_{\lambda}(g\chi_{r_0/4})||_{L^2(B_{\delta}(0))} \cdot || \lambda^{-2} \psi_{\lambda}||_{L^2(B_{\delta}(0))} \\
    &\leq c_2 I(D)^{1/2} A_{\psi} ||g\chi_{r_0/4}||_{\dot{H}^2(B_{\delta}(0))} \lambda^{-1} \\
    &\leq 4 B_{\chi}^{(2)} (r_0/4)^{-2} c_2 I(D)^{1/2} A_{\psi} ||g\chi_{r_0/2}||_{\dot{H}^2(B_{\delta}(0))} \lambda^{-1} \\
    |\langle NL(g)\chi_{r_0/4}, \lambda^{-2} \psi_{\lambda} \rangle| &\leq D^2 A_{\psi} \int_{|x| \leq r_0/2} r^{-D-2}|r^Dg|^2 \lambda^{-D-2} (1 + r/\lambda)^{-2D} dx \\
    &\leq D^2 A_{\psi} I(D) \int_{B_{\delta}(0)} r^{-4}|g|^2\chi_{r_0/2}dx \\
    &\leq D^2 A_{\psi} I(D) ||r^{-2}g\chi_{r_0/2}||_{L^2(B_{\delta}(0))}^2 \\
    &\leq D^2 A_{\psi} I(D)C_R^2||g \chi_{r_0/2}||_{\dot{H}^2(B_{\delta}(0))}^2 \\
    |\langle [\Delta,\chi_{r_0/4}]g, \lambda^{-2} \psi_{\lambda} \rangle| &= |\langle (\Delta \chi_{r_0/4}) g + 2 \nabla \chi_{r_0/4} \cdot \nabla g, \lambda^{-2} \psi_{\lambda} \rangle|\\
    &\leq 2A_{\psi} B_{\chi}^{(2)}(D+1)(r_0/2)^{-2} \\
    &\cdot \int_{r_0/4 \leq |x| \leq r_0/2} (|g| + |g_r|)  \lambda^{-D-2} (1 + r/\lambda)^{-2D} dx \\
     &\leq 2A_{\psi} B_{\chi}^{(2)}(D+1)(r_0/2)^{-(D+4)} \\
     &\cdot \int_{r_0/4 \leq |x| \leq r_0/2}  r^{D+2}\left(\frac{|g|}{r^2} + \frac{|g_r|}{r}\right)  \lambda^{-D-2} (1 + r/\lambda)^{-2D} dx \\
          &\leq 2A_{\psi} B_{\chi}^{(2)}(D+1)(r_0/2)^{-(D+4)}I(D) \\
          &\cdot \int_{r_0/4 \leq |x| \leq r_0/2}  r^{D+2}\left(\frac{|g|}{r^2} + \frac{|g_r|}{r}\right)  dx \\
           &\leq 4A_{\psi} B_{\chi}^{(2)}(D+1)(r_0/2)^{-(D+4)}I(D)C_R \\
           &\cdot ||r^{D+2}||_{L^2(r_0/4 \leq |x| \leq r_0/2)} \cdot ||g \chi_{r_0/2}||_{\dot{H}^2(B_{\delta}(0))}.
\end{align*}
Similarly, if $\psi$ decays like $|\psi| \leq A_{\psi} \langle y \rangle^{-D-2-\varepsilon}$ (e.g., $|\psi| \lesssim \langle y\rangle^{-2D}$ with $D \geq 3$ or $\psi$ is compactly supported), then
\begin{align}
        |\langle g \chi_{r_0/4}, \lambda^{-2}\psi_{\lambda} \rangle| &\leq \int_{|x| \leq r_0/2} |g| \lambda^{-D-2} \psi(r/\lambda)dx \notag \\
     &\leq   ||r^{-1} g \chi_{r_0/2}||_{L^2(B_{\delta}(0))} \cdot ||r \lambda^{-2} \psi_{\lambda} ||_{L^2(B_{\delta}(0))} \notag \\
    &\leq C_R || g \chi_{r_0/2}||_{\dot{H}^1(B_{\delta}(0))} \cdot ||r \psi ||_{L^2(\mathbb R^{2D+2})}  \label{eq: estimate on g chi lambda^-2 psi}
\end{align}
\end{lemma}

Using Lemma \ref{lemma: inner product inequalities with psi}, we first prove (\ref{eq: modulation, weaker form}). Consider (\ref{eq: modulation estimate, general case}) with $\psi = Z$, which is compactly supported. Observe that $\partial_t \langle g \chi_{r_0/4}, Z_{\lambda} \rangle = 0$ by orthogonality.  Assuming (\ref{eq: smallness of inner product for g not depending on n}), it holds that
$$
| \langle g\chi_{r_0/4},\lambda^{-2} (\Lambda + 2) Z_{\lambda} \rangle| \leq  1/4
$$ 
by (\ref{eq: estimate on g chi lambda^-2 psi}). Furthermore, if  $r_0/(4\lambda) > R_0$, by definition of $Z = ||\Lambda W||^{-1}_{L^2(\mathbb R^{2D+2})}(\Lambda W) \chi_{r \leq R_0}$ with $R_0 \gg 1$ chosen so that $\langle Z, \Lambda W \rangle_{L^2(\mathbb R^{2D+2})} \geq 1/2$, we deduce that
\begin{align*}
    \left|   \frac{\lambda_t}{\lambda}  \right| \cdot |\langle g \chi_{r_0/4}, (\Lambda + 2) Z_{\lambda} \rangle -  \langle   \chi_{r_0/4}(\Lambda W)_{\lambda}, Z_{\lambda} \rangle| &\geq   \left|   \frac{\lambda_t}{\lambda}  \right| \left(\frac{1}{2} - | \langle g \chi_{r_0/4}, (\Lambda + 2) Z_{\lambda} \rangle | \right) \lambda^2\\
      &\geq  \frac{1}{4} \cdot \left|   \frac{\lambda_t}{\lambda}  \right| \cdot \lambda^2.
\end{align*}
Lemma \ref{lemma: inner product inequalities with psi} applied with $\psi = Z$ implies that the commutator part, the nonlinear part and the $H_{\lambda}$ part are upper bounded by a term
$$
  C(\delta,r_0,D,\chi,M,H,Z)\left( \frac{||g \chi_{r_0/2}||_{\dot{H}^2}}{\lambda} + ||g \chi_{r_0/2}||_{\dot{H}^2}^2 \right). 
$$
Plugging those in (\ref{eq: modulation estimate, general case}) proves (\ref{eq: modulation, weaker form}).

 Next, we prove (\ref{eq: modulation, more precise}).  Assuming (\ref{eq: smallness of inner product for g not depending on n}), it holds that
\begin{equation}
    | \langle g\chi_{r_0/4},\lambda^{-2} (\Lambda + 2) (\Lambda W)_{\lambda} \rangle| \leq   ||\Lambda W||_{L^2(B_{\delta}(0))}^2/2 \label{eq:D_1 bounded}
\end{equation}
by (\ref{eq: estimate on g chi lambda^-2 psi}). This also holds for $D = 2$ as $(\Lambda + 2)\Lambda W$ has an improved $\langle y \rangle^{-2D-2}$ decay in that specific equivariance class. Moreover,
\begin{align*}
    \langle \chi_{r_0/4} (\Lambda W)_{\lambda}, \lambda^{-2} (\Lambda W)_{\lambda} \rangle &= ||\Lambda W||_{L^2(\mathbb R^{2D+2})}^2 - \int_{\mathbb R^{2D+2}}(1- \chi_{r_0/(4\lambda)}) |\Lambda W|^2 dy, \\
    \int_{\mathbb R^{2D+2}}|1- \chi_{r_0/(4\lambda)}| \cdot |\Lambda W|^2 dy &\leq  \left|\mathbb S^{2D+1} \right|\int_{r \geq r_0/(4\lambda)} \frac{4D^2}{(1+r^{2D})^2}r^{2D+1} dr \\
    &\leq  \left|\mathbb S^{2D+1} \right| (2D-2)^{-1} 4D^2 \left( \frac{r_0}{4\lambda} \right)^{-(2D-2)} \\
    &\leq  ||\Lambda W||_{L^2(\mathbb R^{2D+2})}^2/4.
\end{align*}
Hence, the fraction 
$$
\frac{1}{\langle \chi_{r_0/4} (\Lambda W)_{\lambda}, \lambda^{-2} (\Lambda W)_{\lambda} \rangle-\langle g\chi_{r_0/4},\lambda^{-2} (\Lambda + 2) (\Lambda W)_{\lambda} \rangle } \leq \frac{4}{ ||\Lambda W||_{L^2(B_{r_0}(0))}^2},
$$
is well-defined. Multiplying by this fraction in (\ref{eq: modulation estimate, crucial case}), we have proved so far that
\begin{align}
   \left| \frac{\lambda_t}{\lambda} + \tilde{\omega}(t) \right| &\leq C(\Lambda W, r_0) \cdot |\langle - [\Delta,\chi_{r_0/4}]g + NL(g)\chi_{r_0/4} , \lambda^{-2}(\Lambda W)_{\lambda} \rangle | \notag \\
   &\leq  C(\delta,r_0,D,\chi,M,H,Z) ||g \chi_{r_0/2}||_{\dot{H}^2}^2, \label{eq: bound on N(t)} \\
   \tilde{\omega}(t) &:= -\frac{\partial_t \langle g \chi_{r_0/4},\lambda^{-2} (\Lambda W)_{\lambda} \rangle}{ \langle \chi_{r_0/4} (\Lambda W)_{\lambda}, \lambda^{-2} (\Lambda W)_{\lambda} \rangle -\langle g\chi_{r_0/4}, \lambda^{-2} (\Lambda + 2) (\Lambda W)_{\lambda} \rangle }, \notag
\end{align}
 as Lemma \ref{lemma: inner product inequalities with psi} applied with $\psi = (\Lambda W)$ implies that the commutator and nonlinear parts are upper bounded. 
 
 We focus on $D \geq 3$ first. In order to prove (\ref{eq: modulation, more precise}), it suffices to show that $\tilde{\omega}(t) = \partial_t \mathcal{O}(1) + \mathcal{O}(||g \chi_{r_0/2}||_{\dot{H}^2}^2 + 1 )$. 
 
 Let $(\Lambda + 2)\Lambda W = (2-D)\Lambda W + R$, $R \lesssim \langle y \rangle^{-2D-2}$ when $D \geq 3$ and $h : [-2/3,2/3] \to \mathbb R$, $|h(A)| \lesssim A^2$,
$$
h(A) = \int_{0}^A \frac{a}{(1+ a)^2}da.
$$
Letting
\begin{align*}
    N(t) &= \langle g \chi_{r_0/4}, \lambda^{-2}(\Lambda W)_{\lambda} \rangle \\
    D_1(t) &= -\langle g \chi_{r_0/4}, \lambda^{-2} (\Lambda + 2) (\Lambda W)_{\lambda} \rangle, \quad D_2(t) = \langle \chi_{r_0/4} (\Lambda W)_{\lambda}, \lambda^{-2} (\Lambda W)_{\lambda} \rangle,
\end{align*}
one has
\begin{align*}
    \frac{\partial_t N(t)}{D_1(t) + D_2(t)} &= \partial_t \left[ \frac{N(t)}{D_1(t) + D_2(t)} + \frac{1}{D-2} h \left( \frac{D_1(t)}{D_2(t)} \right) \right] \\
    &+ \frac{\langle g \chi_{r_0/4}, \lambda^{-2}R_{\lambda} \rangle  \partial_t D_1(t)}{(D-2)(D_1(t) + D_2(t))^2} \\
    &+ \frac{  \partial_t D_2(t)}{(D-2)(D_1(t) + D_2(t))^2} \left[ \langle g \chi_{r_0/4}, \lambda^{-2}R_{\lambda} \rangle   +D_1(t) + \frac{D_1(t)^2}{D_2(t)} \right] \\
    &= (I) + (II) + (III),
\end{align*}
where
\begin{align}
    |N(t)| = |\langle g \chi_{r_0/4}, \lambda^{-2}(\Lambda W)_{\lambda} \rangle| &\lesssim_{\Lambda W, D, \delta} || g \chi_{r_0/2}||_{\dot{H}^1(B_{\delta}(0))} \leq  C(\delta,r_0,D,\chi,M,H,Z)  \label{eq: upper bound on innerproduct of g chi r_0/4}
\end{align}
by (\ref{eq: estimate on g chi lambda^-2 psi}). So the first term $(I)$ is $\partial_t \mathcal{O}(1)$.

As for (II), we first deal with $\partial D_1(t)$. Consider  (\ref{eq: modulation estimate, general case}) applied with $\psi = (\Lambda + 2)(\Lambda W)$, together with Lemma \ref{lemma: inner product inequalities with psi} applied with $\psi = (\Lambda + 2)(\Lambda W)$. We deduce that
\begin{align*}
    |\partial_t D_1(t)| &\leq \left| \frac{\lambda_t}{\lambda} \right| \left( ||\Lambda W||_{L^2(\mathbb R^{2D+2})}^2 + C_R || g \chi_{r_0/2}||_{\dot{H}^1(B_{\delta}(0))} \cdot ||r \psi ||_{L^2(\mathbb R^{2D+2})} \right)\\
&+  C(\delta,r_0,D,\chi,M,H,Z)\left( \frac{||g \chi_{r_0/2}||_{\dot{H}^2}}{\lambda} + ||g \chi_{r_0/2}||_{\dot{H}^2}^2 \right). 
\end{align*}
which proves
$$
 |\partial_t D_1(t)| \leq   C(\delta,r_0,D,\chi,M,H,Z) \left( \frac{||g \chi_{r_0/2}||_{\dot{H}^2}}{\lambda} + ||g \chi_{r_0/2}||_{\dot{H}^2}^2 \right)
$$
using (\ref{eq: smallness of inner product for g not depending on n}) and (\ref{eq: modulation, weaker form}). After multiplying by $\langle g \chi_{r_0/4}, \lambda^{-2}R_{\lambda} \rangle$ in (II), the $\lambda^{-1}$ factor is cancelled out. Indeed, observe that 
\begin{align}
   | \langle g \chi_{r_0/4}, \lambda^{-2}R_{\lambda} \rangle |&\lesssim_{R} \int_{|x| \leq r_0/4} |g| \lambda^{-D-2} (1 + r/\lambda)^{-2D-2} dx \notag \\
    &\lesssim_{R} \lambda ||r^{-2} g \chi_{r_0/2}||_{L^2(B_{\delta}(0))} \cdot ||r^2 \lambda^{-3} (1+r/\lambda)^{-2D-2} ||_{L^2(B_{\delta}(0))} \notag \\
    &\lesssim_{R, D, H, \delta} \lambda || g \chi_{r_0/2}||_{\dot{H}^2(B_{\delta}(0))}  \notag \\
   | \langle g \chi_{r_0/4}, \lambda^{-2}R_{\lambda} \rangle |&\lesssim_{R} \int_{|x| \leq r_0/4} |g| \lambda^{-D-2} (1 + r/\lambda)^{-2D-2} dx \notag \\
    &\lesssim_{R} \lambda ||r^{-1} g \chi_{r_0/2}||_{L^2(B_{\delta}(0))} \cdot ||r \lambda^{-3} (1+r/\lambda)_{\lambda}^{-2D-2} ||_{L^2(B_{\delta}(0))} \notag \\
    &\lesssim_{R, D, H, \delta} 1 \label{eq: inner product with R estimate}
\end{align}
thanks to the improved decay of $R$, as well as (\ref{eq: smallness of inner product for g not depending on n}). This shows that 
$$
|\langle g \chi_{r_0/4}, \lambda^{-2}R_{\lambda} \rangle| \left( \frac{||g \chi_{r_0/2}||_{\dot{H}^2}}{\lambda} + ||g \chi_{r_0/2}||_{\dot{H}^2}^2 \right) \leq  C \cdot \left(  ||g \chi_{r_0/2}||_{\dot{H}^2}^2 + ||g \chi_{r_0/2}||_{\dot{H}^2}^2 \right).
$$
Hence, (II) is $\mathcal{O}( ||g \chi_{r_0/2}||_{\dot{H}^2}^2)$.

Finally, note that the only problematic term in (III) is $\partial_t D_2(t)$, as the other terms are uniformly bounded by a constant, see (\ref{eq:D_1 bounded}) and (\ref{eq: inner product with R estimate}). A direct computation, together with (\ref{eq: modulation, weaker form}), shows that
\begin{equation}
    |\partial_t D_2(t)| \lesssim_{r_0, D, \chi} \lambda^{2D-2} \left| \frac{\lambda_t}{\lambda} \right| \lesssim_{r_0, D, \chi}  \lambda^{2D-2} \left( \frac{||g \chi_{r_0/2}||_{\dot{H}^2}}{\lambda} + ||g \chi_{r_0/2}||_{\dot{H}^2}^2 \right), \label{eq: upper bound partial_t D_2}
\end{equation}
which implies that (III) is $\mathcal{O}(||g \chi_{r_0/2}||_{\dot{H}^2}^2 + 1 )$.

In the $D=2$ case, we cannot divide by $(D-2)$. However, $(\Lambda + 2)\Lambda W = R$ has an improved decay rate that we can take advantage of. Let
$$
E(t) = \frac{\lambda_t}{\lambda} + \tilde{\omega}(t).
$$
Then
\begin{align*}
    |\partial_t N(t)| &=\left| (D_1(t) + D_2(t)) \left(E(t) - \frac{\lambda_t}{\lambda}\right)\right| \\
&\leq C(\delta,r_0,D,\chi,M,H,Z)\left( \frac{||g \chi_{r_0/2}||_{\dot{H}^2}}{\lambda} + ||g \chi_{r_0/2}||_{\dot{H}^2}^2 \right)
\end{align*}
using (\ref{eq: modulation, weaker form}), (\ref{eq:D_1 bounded}) and (\ref{eq: bound on N(t)}). Write
$$
\frac{\partial_t N(t)}{D_1(t) + D_2(t)} = \frac{\partial_t N(t)}{D_2(t)} - \frac{D_1(t) \partial_t N(t)}{D_2(t)(D_1(t) + D_2(t))} = (I) + (II).
$$
Observing that $D_1(t) = -\langle g \chi_{r_0/4},\lambda^{-2}(\Lambda + 2)(\Lambda W)_{\lambda} \rangle$ when $D = 2$, (II) is $\mathcal{O}( ||g \chi_{r_0/2}||_{\dot{H}^2}^2)$ by (\ref{eq: inner product with R estimate}). Next, write
$$
(I) = \frac{\partial_t N(t)}{D_2(t)} = \partial_t \left[ \frac{N(t)}{D_2(t)} \right] + \frac{N(t) \partial_t D_2(t)}{D_2(t)^2}.
$$
When $D = 2$, one cannot use   (\ref{eq: upper bound on innerproduct of g chi r_0/4}). Instead, we have the following logarithmic estimate
\begin{align*}
    |N(t)| = |\langle g \chi_{r_0/4}, \lambda^{-2}(\Lambda W)_{\lambda} \rangle| &\lesssim_{\Lambda W, D, \delta} || g \chi_{r_0/2}||_{\dot{H}^1(B_{\delta}(0))} \cdot ||r \Lambda W||_{L^2(B_{\lambda^{-1} \delta}(0))} \\
     &\lesssim_{\Lambda W, D, Z, \delta, r_0}  \sqrt{1 + |\log(\lambda)|}
\end{align*}
using (\ref{eq: smallness of inner product for g not depending on n}). Then 
$$
\partial_t \left[ \frac{N(t)}{D_2(t)} \right]  = \partial_t O(\sqrt{1 + |\log(\lambda)|}),
$$
while 
\begin{align*}
    \frac{N(t) \partial_t D_2(t)}{D_2(t)^2} &= \sqrt{1 + |\log(\lambda)|} \mathcal{O} \left[  \lambda^{2} \left( \frac{||g \chi_{r_0/2}||_{\dot{H}^2}}{\lambda} + ||g \chi_{r_0/2}||_{\dot{H}^2}^2 \right) \right]  \\
    &=  \ \mathcal{O} \left[   \left(||g \chi_{r_0/2}||_{\dot{H}^2} + ||g \chi_{r_0/2}||_{\dot{H}^2}^2 \right) \right] =   \mathcal{O} \left[   1 + ||g \chi_{r_0/2}||_{\dot{H}^2}^2  \right] 
\end{align*}
using $\lambda \sqrt{1 + |\log(\lambda)|} \lesssim 1$ when $\lambda < 1$ and (\ref{eq: upper bound partial_t D_2}).  This finishes the proof.
\end{proof}

\section{Applying modulation estimates: a conditional contradiction}\label{sec: applying modulation}
Under the previous assumptions, the modulation estimates  (\ref{eq: modulation, weaker form}) and (\ref{eq: modulation, more precise}), together with (\ref{eq: upper bound of g(t) in terms of Delta u}) and (\ref{eq: upper bound of g(t) in terms of u_t}), rewrite as
\begin{align*}
        \left| \log(\lambda(t))-\log(\lambda(s)) \right| &\leq I(t,s) + C \int_{t}^{s}  \frac{|| g(a)\chi_{r_0/2}||_{\dot{H}^2(B_{\delta}(0))} }{\lambda(a)}da, \quad t,s \in [t_0,t_1] , \notag \\
        I(t,s) &= \left| \int_t^s 1 + ||g(a) \chi_{r_0/2}||_{\dot{H}^2(B_{\delta}(0))}^2 da \right| = o_{t \to T, s \to T}(1)
\end{align*}
and
\begin{align*}
       \left| \log(\lambda(t))-\log(\lambda(s)) \right| &\leq |\omega(t)-\omega(s)| + C\int_{t}^{s}1 + ||g(a)\chi_{r_0/2}||^2_{\dot{H}^2(B_{\delta}(0))} da \notag \\
        &\leq |\omega(t)-\omega(s)|  +  o_{t \to T,s \to T}(1),
\end{align*}
where the small $o(\cdot)$ function, as well as $C$, depend on $u,\delta,r_0,D,\chi,M,H,Z$, but not on the interval, the scale $\lambda(t)$ or the remainder $g$.

Recall the profile decomposition for $v$ and $u$ along a sequence of times combined with the single-bubble result:
\begin{equation}
     \sup_{t \in [t_0^{(n)},t_1^{(n)}]} ||u(\cdot,t) - u_*(\cdot) -  W_{\lambda_0^{(n)}}||_{H^1(B_{\delta}(0))} \to 0, \label{eq: single bubble}
\end{equation}
where we used the normalization $v(0,t) = 0$, $t < T$, $v(0,T) = \pi$ and the single-bubble result. On each interval, one obtains from Orbital Stability (Section \ref{sec:orbital stability}) a $C^1$-scale $\lambda_n(t)$ (technically, it should be $\lambda_n(u(t))$) for which $|\lambda_n(t)| \sim \lambda_0^{(n)} \to 0$ and a remainder $\tilde{g}_n$ bounded by a multiple of $\pi$, $u = W_{\lambda_n} + \tilde{g}_n$, satisfying the orthogonality condition $$
\langle \tilde{g}_n(t,\cdot), Z_{\lambda_n}(t)\rangle_{L^2(B_{\delta}(0))} = 0
$$
and the estimates (\ref{eq: smallness of inner product with g_n}).

For each $n$, one deduces a $L^2$-estimate for $\tilde{g}_n$ which is independent of $n$. Letting $\tilde{g}(s)$ (resp. $\lambda(t)$) be defined as $\tilde{g}_n(s)$ (resp. $\lambda_n(t)$) on $[t_0^{(n)},t_1^{(n)}]$, then
$$
||\tilde{g}\chi_{r_0/2}||_{\dot{H}^2(B_{\delta}(0))} \in L^2_t\left( \bigcup_{n} [t_0^{(n)},t_1^{(n)}] \right), \quad I(t,s) = o_{\substack{t,s \in \cup_n [t_0^{(n)},t_1^{(n)}] \\ t \to T, s \to T } }(1)
$$
by (\ref{eq: upper bound of g(t) in terms of Delta u}), (\ref{eq: upper bound of g(t) in terms of u_t}), and one deduces
\begin{align}
        \left| \log(\lambda(t))-\log(\lambda(s)) \right| &\leq I(t,s) + C \int_{t}^{s}  \frac{|| \tilde{g}(a)\chi_{r_0/2}||_{\dot{H}^2(B_{\delta}(0))} }{\lambda(a)}da  \label{eq: modulation estimates weaker} \\ 
         \left| \log(\lambda(t))-\log(\lambda(s)) \right| &\leq |\omega(t)-\omega(s)| +  o_{t \to T,s \to T}(1) \label{eq: modulation estimates improved}
\end{align}
whenever $t, s \in [t_0^{(n)},t_1^{(n)}]$. Again, the small $o(\cdot)$ function, as well as $C$, depend on $u,\delta,r_0,D,\chi,M,H,Z$, but not on the interval, the scale $\lambda(t)$ or the remainder $\tilde{g}$.

The problem is that these functions only exist on some small time-interval, i.e., these bounds are valid for $s,t \in [t_0^{(n)},t_1^{(n)}]$.  If these bounds were to hold for any $s,t$ close to $T$, then fixing $s$ in (\ref{eq: modulation estimates improved}), we have
\begin{equation}
    |\log \lambda(t) | \leq
|\log \lambda(s) | + |\omega(s) | + o_{t \to T,s \to T}(1) + C(1 +\sqrt{|\log \lambda(t)|}) \label{eq: final contradiction}
\end{equation}
Letting $t \to T$, one has $\lambda(t) \to 0$ in case of finite-time blow-up at $T < +\infty$, which contradicts (\ref{eq: final contradiction}). Hence, the solution must be global for $D \geq 2$. 

For this last contradiction to be rigorous, we need a final argument to prove that, as $t \to T$, the solution $v$ stays continuously close to the family of rescaled bubbles, i.e., that the scale $\lambda(t)$ is defined continuously on $[t_0^{(1)},T)$. We take inspiration from Jendrej-Lawrie's argument (\cite{jendrej}) and use the notion of collision intervals.

 \section{Collision Intervals}\label{sec:collision intervals}
 In the following, we work at the level of $v(r,t)$ again.

 Let
\begin{align*}
    d_1(t,\rho) &= \inf_{\lambda} \left( ||v(t,\cdot)-v(T,\cdot)-Q_{\lambda}+\pi||_{H^1(B_{\delta}(0) \setminus B_{\rho}(0))} + \frac{\rho}{\lambda} + \frac{\lambda}{\sqrt{T-t}} \right) \\
    d_0(t,\rho) &= ||v(t,\cdot)-v(T,\cdot)||_{H^1(B_{\delta}(0) \setminus B_{\rho}(0))} + \frac{\rho}{\sqrt{T-t}}.
\end{align*}
The abstract profile decomposition, the Palais-Smale condition and the above analysis tell us that there is $t_0^{(n)} \to T$ for which $\lim \limits_{n \to +\infty} d_1(t_0^{(n)},0) = 0$. Moreover, there exists $\rho(t)$ (constructed  from (\ref{eq: no concentration at self-similar scale})) using a diagonal argument) for which $\lim \limits_{t \to T} d_0(t,\rho(t)) = 0$.

\begin{theorem}[Continuous bubbling]\label{thm: continuous bubbling}
    Let $D \geq 2$. The solution stays continuously close to the single-bubble, i.e., $$\lim \limits_{t \to T}d_1(t,0) = 0.$$
\end{theorem}

The proof will be by contradiction, in the same spirit as Section 5 in \cite{jendrej}.

\begin{definition}[Collision Interval]
    We say that an interval $[a,b]$ is a collision-interval $C_i([a,b],\varepsilon,\eta,\rho(t))$ for the parameters $\varepsilon < \eta, \rho(t)$ if $d_1(a,0) \leq \varepsilon, d_1(b,0) \geq \eta$ and
    $$
\sup_{t \in[a,b]}d_i(t,\rho(t)) \leq \varepsilon.
    $$
\end{definition}

\begin{lemma} \label{lemma: non-continuous decomposition implies collision}
    If $\lim \limits_{t \to T} d_1(t,0) \neq 0$, there exists $\varepsilon_n \to 0, \eta > 0, a_n \to T, b_n \to T$, $a_1 < b_1 < a_2 < b_2$ ... and $\rho(t)$ defined on each $[a_n,b_n]$ for which $[a_n,b_n]$ is a collision-interval $C_0([a_n,b_n],\varepsilon_n, \eta, \rho(t))$. Moreover, there are no such sequences for which $[a_n,b_n]$ is a collision-interval $C_1$.
\end{lemma}

\begin{proof}
    Let $a_n = t_0^{(n)} \to T$ with $ \varepsilon_n := d_1(t_0^{(n)},0) \to 0$ and assume that $\lim \limits_{n \to +\infty} d_1(b_n,0) \geq \eta > 0$ along some sequence $b_n \to T$. Up to taking subsequences, one can assume that $a_1 < b_1 < a_2 < b_2$ ... and $\varepsilon_n$ is decreasing. Observe that
    $$
    \lim \limits_{n \to +\infty} \sup_{t \in [a_n,b_n]} d_0(t,\rho(t)) = 0
    $$ 
    for $\rho(t)$ obtained using (\ref{eq: no concentration at self-similar scale}). Choose $n_1 < n_2 < ... < n_k$ always larger than the previous index and satisfying
    $$
    \sup_{t \in [a_{n_{k}},b_{n_{k}}]} d_0(t,\rho(t)) \leq \varepsilon_{k}, \quad  d_1(a_{n_{k}},0) = \varepsilon_{n_{k}} \leq \varepsilon_k.
    $$
    Passing to this subsequence (which we still denote with index $n$), we have found a collision-interval $C_0([a_n,b_n],\varepsilon_n,\eta,\rho(t))$.

    Assume now for a contradiction that there is a sequence of collision intervals $C_1([a_n,b_n],\varepsilon_n, \eta, \rho_0(t))$. For each $b_n, \rho_0(b_n)$, there is some $\rho_0(b_n) \lesssim \varepsilon_n \lambda_n \lesssim \varepsilon_n^2 \sqrt{T-b_n}$ for which
    $$
     d_1(b_n,\rho_0(b_n))  \simeq ||v(b_n,\cdot)-v(T,\cdot)-Q_{\lambda_n}+\pi||_{H^1(B_{\delta}(0) \setminus B_{\rho_0(b_n)}(0))}  \lesssim \varepsilon_n.
    $$
    From (\ref{eq: existence limit of energy at time T}), we deduce that
    $$
    ||v(b_n,\cdot)||_{\dot{H}^1(B_{\delta}(0))}^2  \to ||Q||_{\dot{H}^1(\mathbb R^{2})}^2 + ||v(T,\cdot)||_{\dot{H}^1(B_{\delta}(0))}^2.
    $$
    Moreover, observe that $d_1(b_n,\rho_0(b_n)) \to 0$ implies
    \begin{align*}
        \lim \limits_{n \to +\infty} ||v(b_n,\cdot)||_{\dot{H}^1(B_{\delta}(0) \setminus B_{\rho_0(b_n)}(0))}^2 &= \lim \limits_{n \to +\infty} ||v(T,\cdot)-Q_{\lambda_n}||_{\dot{H}^1(B_{\delta}(0) \setminus B_{\rho_0(b_n)}(0))}^2 \\
        &= ||Q||_{\dot{H}^1(\mathbb R^{2})}^2 + ||v(T,\cdot)||_{\dot{H}^1(B_{\delta}(0))}^2
    \end{align*}
    by reverse triangle inequality and orthogonality of scales. Hence, 
    $$
    ||v(b_n,\cdot)||_{\dot{H}^1(B_{\rho_0(b_n)}(0))}^2  \to 0.
    $$
    Lemma \ref{lemma:closeness to l0 pi} implies that along subsequence, there is $l_n \in \mathbb Z$ for which $|v(r,b_n) - l_n \pi| = o_n(1)$ on $[0, \rho_0(b_n)]$. But evaluating at $r = 0$ shows that $l_n = 0$. In particular,
    $$
    ||v(b_n,\cdot)||_{H^1(B_{\rho_0(b_n)}(0))}^2  \to 0.
    $$
    The same holds for $v(T,\cdot)-\pi$ and $Q_{\lambda_n}$ on $B_{\rho_0(b_n)}(0)$, which means that
    $$
||v(b_n,\cdot) - v(T,\cdot) - Q_{\lambda_n} + \pi||_{H^1(B_{\rho_0(b_n)}(0))}^2  \to 0,
    $$
    while
$$
d_1(b_n,\rho_0(b_n))^2 \simeq  ||v(b_n,\cdot) -v(T,\cdot) - Q_{\lambda_n} + \pi||_{H^1(B_{\delta}(0) \setminus B_{\rho_0(b_n)}(0))}^2  \to 0,
    $$
    which contradicts $d_1(b_n,0) \geq \eta$.
\end{proof}

In the following, we assume that $\lim \limits_{t \to T}d_1(t,0) \neq 0$ and reach a contradiction. Consider a sequence of collision intervals $C_0([a_n,b_n],\varepsilon_n,\eta,\rho(t))$ as in Lemma \ref{lemma: non-continuous decomposition implies collision}.

By selecting a smaller $b_n$ if needed, one may always assume that $d_1(b_n,0) = \eta$ and $d_1(c,0) \leq \eta$ for all $c \in [a_n,b_n]$ in our sequence of collision intervals. Note also that for any $\lambda> 0$,
$$
v(t,0) - v(T,0) - Q_{\lambda}(0) + \pi = 0.
$$
We may also assume that $C(D,\delta)\eta$ is smaller than the $\delta_1$ from Orbital Stability, where $C(D,\delta)$ is the constant from Lemma \ref{lemma:equivalence of remainder convergence}. In particular, one can apply Orbital Stability on each $[a_n,b_n]$ and obtain a $C^1$-scale $\lambda(t)$ defined on the union $\cup_n [a_n,b_n]$ on each interval, since $\sup \limits_{[a_n,b_n]} d_1(t,0) \leq \eta$. This scale satisfies the modulation estimates (\ref{eq: modulation estimates weaker}).

\begin{lemma}\label{lemma: lambda(t) and lambda(s) become similar in the limit}
Let $1/2 > \tilde{\varepsilon} > 0$ be fixed. There exists $n_0( \tilde{\varepsilon} ) \gg 1$ such that for any function $0 \leq \theta(t) \leq 1$ defined on $\cup_{n}[a_n,b_n]$, if $n \geq n_0$ and $[s, s + \theta(s)\lambda(s)^2]$ lies inside $[a_n,b_n]$, then 
$$
(1- \tilde{\varepsilon} ) \lambda(s) < \lambda(t) < (1+ \tilde{\varepsilon} ) \lambda(s), \quad t \in  [s, s + \theta(s)\lambda(s)^2].
$$
Moreover,
    $$
 \sup_{s \leq t \leq s + \theta(s)\lambda(s)^2} \left| \log \frac{\lambda(t)}{\lambda(s)} \right| \leq \tilde{\varepsilon}.
    $$
\end{lemma}

\begin{proof}
   Let $\tau \in (s, s + \theta(s)\lambda(s)^2]$ be maximal with
    $$
    \lambda(t) \in \left( (1- \tilde{\varepsilon} ) \lambda(s), (1+ \tilde{\varepsilon} ) \lambda(s) \right), s \leq t \leq \tau
    $$
    We show that $\tau = s + \theta(s)\lambda(s)^2$ by showing that 
    $$
    \lambda(\tau) \in \left( (1- \tilde{\varepsilon} ) \lambda(s), (1+ \tilde{\varepsilon} ) \lambda(s) \right)
    $$
    if $n \geq n_0( \tilde{\varepsilon} )$. When $s \leq t \leq \tau$, (\ref{eq: modulation estimates weaker}) implies
\begin{align*}
        \left|\log \frac{\lambda(t)}{\lambda(s)}\right| &\leq  I(t,s)  + C \int_{s}^{t}  \frac{|| g(a)\chi_{r_0/2}||_{\dot{H}^2(B_{\delta}(0))} }{\lambda(a)}da \\
         &\leq  I(t,s)  +\frac{(1-\tilde{\varepsilon})^{-1}C}{\lambda(s)} \sqrt{t-s} \left( \int_{s}^{t}  || g(a)\chi_{r_0/2}||_{\dot{H}^2(B_{\delta}(0))}^2 da \right)^{1/2} \\
         &\leq I(t,s)  + 2C \sqrt{I(t,s)} = o_{n \to +\infty}(1),
\end{align*}
where the $o_{n\to +\infty}(1)$ depends only on the interval $[a_n,b_n]$.

    If $n \geq n_0( \tilde{\varepsilon})$ is large enough, then $o_{n \to +\infty}(1) e^{|o_{n \to +\infty}(1)|} <  \tilde{\varepsilon}$ and
\begin{align*}
     \left| \frac{\lambda(t)}{\lambda(s)} - 1\right| &= \left|\mathrm{exp}\left( \log \left( \frac{\lambda(t)}{\lambda(s)} \right) \right) - e^0 \right|  
     \\
     &\leq 
     \left|\log \left( \frac{\lambda(t)}{\lambda(s)} \right)  \right| \mathrm{exp} \left( \left| \log \frac{\lambda(t)}{\lambda(s)}  \right| \right) \\
     &<  \tilde{\varepsilon}, \quad s \leq t \leq \tau,
\end{align*}
     which implies the appropriate bounds for $\lambda(\tau)$.
\end{proof}

\begin{lemma}\label{lemma:collision interval length estimate}
   For all $0 < \tilde{\eta} < \eta$, there is $\varepsilon \in (0,\tilde{\eta})$, $C > 0$, $n_0 \in \mathbb N$ for which if $[c,d] \subset [a_n,b_n]$ with $d_1(c,0) \leq \varepsilon, d_1(d,0) \geq \tilde{\eta}$, $n \geq n_0$, then 
    $$
(d-c)^{\frac{1}{2}} \geq C \lambda(c).
    $$
\end{lemma}

\begin{proof}
    Assume not. Then there is $\eta > \tilde{\eta} > 0$ and sequences $\varepsilon_n \to 0, C_n \to 0$ for which the property is not true on a subinterval $[c_n,d_n] \subset [a_n,b_n]$, i.e.
    $$
    (d_n-c_n)^{\frac{1}{2}} \leq C_n \lambda(c_n).
    $$
    We prove that $d_1(d_n,0) \to 0$, contradicting $d_1(d_n,0) \geq \tilde{\eta}$.
    
        Note that by Orbital Stability, $\lambda(c_n) \sim \lambda_n \ll \sqrt{T - c_n}$, where $\lambda_n$ is a scale for which
    $$
    ||v(c_n,\cdot) - v(T,\cdot) - Q_{\lambda_n} + \pi||_{H^1(B_{\delta}(0))} \lesssim d_1(c_n,0) \to 0.
    $$
    Assume without loss of generality that $n$ is large enough so that $\lambda(c_n) \leq \sqrt{T-c_n}$ and $C_n \leq 1/2$. One finds that
            $$
\lim \limits_{n \to +\infty} \frac{\lambda_n}{\sqrt{T - d_n}}= 0
    $$
    due to
    $$
 1 \leq \frac{\sqrt{T-c_n}}{\sqrt{T-d_n}} \leq \frac{\sqrt{T-d_n + (d_n-c_n)}}{\sqrt{T-d_n}} \leq  1 + \frac{C_n\lambda(c_n)}{\sqrt{T-d_n}} \leq 1 +  \frac{1}{2} \frac{\sqrt{T-c_n}}{\sqrt{T-d_n}}.
    $$
    Let $V_n^0(r) = v(c_n,\lambda_nr) - v(T,\lambda_n r) + \pi$, $V_n(r) = v(d_n,\lambda_nr) - v(T,\lambda_n r) + \pi$, be defined on $(0,\delta \lambda_n^{-1})$ and extended as $0$ on $(\delta \lambda_n^{-1},+\infty)$. 
    
    As $V_n^0$ is uniformly bounded in $L^{\infty}(\mathbb R^2)$, a diagonal argument shows that, up to taking a subsequence, $V_n^0$ has a limit $V_{\infty}^0$ in the sense of distribution $\mathcal{D}'(\mathbb R^2)$ which is a $L^2_{loc}(\mathbb R^2)$-function. By hypothesis, $V_n^0(r) \to Q$ in $\dot{H}^1_{loc}(\mathbb R^2)$, hence $\nabla V_{\infty}^0 = \nabla Q$ in the distributional sense. Moreover, for any $R > 0$,
\begin{align*}
   ||V_n(r) - V_n^0(r)||_{L^2(B_{R}(0))} &= || v(c_n,\lambda_nr) - v(d_n,\lambda_nr) ||_{L^2(B_{R}(0))} \\
   &= \lambda_n^{-1} || v(c_n,\cdot) - v(d_n,\cdot) ||_{L^2(B_{R \lambda_n}(0))}  \\
   &\leq \frac{\sqrt{d_n-c_n}}{\lambda_n} \left( \int_{c_n}^{d_n}||v_t||_{L^2(B_{R\lambda_n}(0))}^2dt\right)^{\frac{1}{2}} \\
   &\lesssim C_n \to 0
\end{align*}
as the $L^2$-norm of $v_t$ is finite by (\ref{eq: finite L^2-norm of v_t}). Hence, $V_n(r)$ has the same distributional limit $V_{\infty}^0$ as $V_n^0$. In particular, $\nabla V_n \to \nabla Q$ in the sense of distribution $\mathcal{D}'(\mathbb R^2)$ as well. By Lemma \ref{lemma: failure of convergence of v(t) to v(T) in norm}, 
\begin{align*}
   ||V_n(r)||_{\dot{H}^1(B_{\delta \lambda_n^{-1}}(0))} 
   &= || v(d_n,\cdot) - v(T,\cdot) ||_{\dot{H}^1(B_{\delta}(0))}  \to ||Q||_{\dot{H}^1(\mathbb R^2)}.
\end{align*}
The boundedness of $||V_n(r)||_{\dot{H}^1(B_{\delta \lambda_n^{-1}}(0))}$  allows to upgrade the distributional convergence 
$$
1_{B_{\delta \lambda_n^{-1}}(0))} \cdot \nabla V_n \to \nabla Q
$$
to a weak $L^2(\mathbb R^2)$-convergence and the convergence of the norms then implies strong $L^2(\mathbb R^2)$ convergence, meaning that 
$$
\lim \limits_{n \to +\infty} ||v(d_n,\lambda_n \cdot) - v(T,\lambda_n \cdot) - Q + \pi||_{\dot{H}^1(B_{\delta \lambda_n^{-1}}(0))} = 0
$$
After rescaling,
$$
\lim \limits_{n \to +\infty} ||v(d_n, \cdot) - v(T, \cdot) - Q_{\lambda_n} + \pi||_{\dot{H}^1(B_{\delta}(0))} = 0
$$
and strong $L^2(B_{\delta}(0))$ convergence follows by dominated convergence as everything is uniformly bounded in $L^{\infty}$. This finishes the proof.
\end{proof}

\begin{corollary}\label{cor: collision interval length exact estimate}
    Let $\varepsilon \in (0,\tilde{\eta})$, $C$, $n_0$ be as in Lemma \ref{lemma:collision interval length estimate}. There exists $[c_n,d_n] \subset [a_n,b_n]$, $n_1\in \mathbb N$ such that for all $n \geq n_1$:
    \begin{align*}
        d_1(t,0) &\geq \varepsilon, \quad t \in [c_n,d_n] \\
        d_n-c_n &= \frac{1}{n}\lambda(c_n)^2 \\
        \frac{1}{2}\lambda(c_n) &\leq \lambda(t) \leq 2 \lambda(c_n),  \quad t \in [c_n,d_n]. 
    \end{align*}
\end{corollary}

\begin{proof}
Choose $n$ large enough so that $\varepsilon_n < \min\{\varepsilon, \eta\}$, $n \geq n_0$ and let $\tilde{b}_n \in (a_n,b_n]$ satisfy $d_1(\tilde{b}_n,0) = \tilde{\eta}$. Let $$
c_n = \sup\{t \in [a_n,\tilde{b}_n] : d_1(t,0) \leq \varepsilon\}.
$$
Lemma \ref{lemma:collision interval length estimate} implies
$$
(\tilde{b}_n-c_n)^{\frac{1}{2}} \geq C\lambda(c_n)
$$
Hence, if we define $$
 d_n := c_n + \frac{1}{n}\lambda(c_n)^2.
    $$
    then for $n$ large enough, $d_n \in [c_n,\tilde{b}_n]$ and Lemma \ref{lemma: lambda(t) and lambda(s) become similar in the limit} implies
    $$
    \frac{1}{2}\lambda(c_n) \leq \lambda(t) \leq 2 \lambda(c_n),  \quad t \in [c_n,d_n].
    $$
    Finally, one has $d_1(t,0) \geq \varepsilon$ on $[c_n,d_n]$ by definition of $c_n$.
\end{proof}

\begin{proposition}\label{prop: lambda(t) v_t estimate on c_n d_n}
    Let $[c_n,d_n]$ be defined as in Corollary \ref{cor: collision interval length exact estimate}. Then  
    $$
\liminf_{n \to +\infty} \inf_{t \in [c_n,d_n]} \lambda(t)||v_t(t)||_{L^2(B_{\delta/2}(0))} > 0.
    $$
\end{proposition}
\begin{proof}
    Assume for a contradiction that there is $s_n \in [c_n,d_n]$ with
    $$
    \lim_{n \to +\infty} \lambda(s_n) ||v_t(s_n)||_{L^2(B_{\delta/2}(0))} = 0.
    $$
    As in (\ref{eq:palais smale condition}), it follows that
    $$
\lim \limits_{n \to +\infty} |E'(v(s_n))[\eta_n]| \to 0, \quad \eta_n = \eta(r/\lambda(s_n)), \quad \eta \in C^{\infty}_c((0,+\infty)),
    $$
    and, up to taking a further subsequence, $v_n(r) = v(\lambda(s_n)r,s_n)$ has a weak limit $\psi \in H^1_{loc}(\mathbb R^2)$ which must be a harmonic map. We show that $\psi$ is non-constant. Recall that $d_1(s_n,0) \leq \eta$, which means that there is $\lambda_n \lesssim \sqrt{T-s_n}$ with 
    $$
||v(s_n,r) - v(T,r) -  Q_{\lambda_n}(r) + \pi||_{H^1(B_{\delta}(0))} \leq 2 d_1(s_n,0) \leq 2\eta
$$
and by Orbital Stability,
    $$
v(s_n,r) = Q_{\lambda(s_n)}(r) + r^D\tilde{g}_n(r) =: Q_{\lambda(s_n)}(r) + e_n(r), \quad r \in [0,\delta]
    $$
   where $\lambda(s_n) \in (e^{-1}\lambda_n,e\lambda_n)$ and
    $$
    ||e_n \chi_{r_0/2}||_{H^1(B_{\delta}(0))} \leq ||Q||_{\dot{H}^1(B_1(0))}/4.
    $$
    Rescale everything:
 \begin{align*}
     v(s_n,\lambda(s_n) r) = Q(r) + e_n(\lambda(s_n) r), \quad r \in [0,\delta/\lambda(s_n)].
 \end{align*}
    For $n$ large enough, $\lambda(s_n) < r_0/2 < \delta$ and by scaling invariance,
    \begin{align*}
        || v(s_n,\lambda(s_n) r) - Q||_{\dot{H}^1(B_{1}(0))}  \leq ||e_n||_{\dot{H}^1(B_{\lambda(s_n)}(0))} \leq ||Q||_{\dot{H}^1(B_1(0))}/4.
    \end{align*}
    Passing to the limit,
    $$
||\psi - Q||_{\dot{H}^1(B_{1}(0))}  \leq ||Q||_{\dot{H}^1(B_1(0))}/4,
    $$
    which implies that $\psi$ is non-constant and of the form $m\pi + 2\arctan(\alpha^Dr^D)$ for some $m \in \mathbb Z$ and $\alpha > 0$. In particular, $||\psi||_{\dot{H}^1(\mathbb R^2)} =||Q||_{\dot{H}^1(\mathbb R^2)}$.

     Let
 \begin{align*}
    V_n(r) = v(s_n,\lambda(s_n) r) -v(T,\lambda(s_n) r)+\pi, \quad r \in [0,\delta/\lambda(s_n)].
 \end{align*}
    and extend $V_n(r)$ as $0$ on $(\delta \lambda(s_n)^{-1},+\infty)$. 
    
    As $V_n$ is uniformly bounded in $L^{\infty}(\mathbb R^2)$, a diagonal argument shows that, up to taking a subsequence, $V_n$ has a limit $V_{\infty}$ in the sense of distribution $\mathcal{D}'(\mathbb R^2)$ which is a $L^2_{loc}(\mathbb R^2)$-function, denoted $V_{\infty}$. As $\pi -v(T,\lambda(s_n) r)\to 0$ on compact sets, $V_{\infty} = \psi$.
    
    By Lemma \ref{lemma: failure of convergence of v(t) to v(T) in norm}, 
\begin{align}
 ||V_n(r)||_{\dot{H}^1(B_{\delta/\lambda(s_n)}(0))} 
   &= || v(s_n,\cdot) - v(T,\cdot) ||_{\dot{H}^1(B_{\delta}(0))}  \to ||Q||_{\dot{H}^1(\mathbb R^2)}. \label{eq: psi non-constant}
\end{align}
The distributional convergence and the uniform $L^2$-bound in (\ref{eq: psi non-constant}) implies weak  $L^2(\mathbb R^2)$-convergence 
$$
1_{B_{\delta/\lambda(s_n)}(0)} \cdot \nabla V_n(r) \to \nabla \psi,$$
The convergence of the norms then implies strong $L^2(\mathbb R^2)$-convergence, meaning that 
$$
\lim \limits_{n \to +\infty} ||v(s_n,\lambda(s_n) \cdot) - v(T,\lambda(s_n)  \cdot) - \psi + \pi||_{\dot{H}^1(B_{\delta \lambda(s_n)^{-1}}(0))} = 0
$$
After rescaling,
$$
\lim \limits_{n \to +\infty} ||v(s_n, \cdot) - v(T, \cdot) - \psi_{\lambda(s_n)} + \pi||_{\dot{H}^1(B_{\delta}(0))} = 0.
$$
 Strong $L^2(B_{\delta}(0))$-convergence
 $$
\lim \limits_{n \to +\infty} ||v(s_n, \cdot) - v(T, \cdot) - [\psi_{\lambda(s_n)} - \psi(0)] + \pi||_{H^1(B_{\delta}(0))} = 0
$$
follows by dominated convergence as everything is uniformly bounded in $L^{\infty}$, where $$\psi(r) - \psi(0) = Q(\alpha r)$$ for some $\alpha > 0$. Replacing $\lambda(s_n)$ by $\lambda(s_n)/\alpha$, one can assume without loss of generality that $\alpha = 1$. Repeating the proof of Corollary \ref{cor: lambda_k ll sqrt T-t_k} with the sequence $s_n$ shows that 
 $\lambda(s_n) \ll \sqrt{T-s_n}$. Hence,
    $$
    \lim \limits_{n \to +\infty} d_1(s_n,0) = 0,
    $$
    contradicting the assumption $d_1(s_n,0) \geq \varepsilon$.

\end{proof}

We now reach the final contradiction. 

\begin{proof}[Proof of Theorem \ref{thm: continuous bubbling}]
Let $c_0 > 0$ and $n_0 \in \mathbb N$ be such that 
$$
\inf_{t \in [c_n,d_n]} \lambda(t)||v_t(t)||_{L^2(B_{\delta/2}(0))} \geq c_0, \quad n \geq n_0,
$$
by Proposition \ref{prop: lambda(t) v_t estimate on c_n d_n}. Using Corollary \ref{cor: collision interval length exact estimate}, 
\begin{align*}
\int_{c_n}^{d_n} ||v_t(t)||_{L^2(B_{\delta/2}(0))}^2 dt &\geq \int_{c_n}^{d_n} \frac{c_0^2}{\lambda(t)^2} dt \\
&\geq  \int_{c_n}^{d_n} \frac{c_0^2}{4\lambda(c_n)^2} dt = \frac{c_0^2}{4n}.
\end{align*}
Summing over $n$, it follows that
$$
\int_{c_1}^{T} ||v_t(t)||_{L^2(B_{\delta/2}(0))}^2 dt = +\infty
$$
diverges, contradicting (\ref{eq: finite L^2-norm of v_t}). Hence, $\lim \limits_{t \to T}d_1(t,0) = 0$.
\end{proof}

\section{Completion of the proof of Theorem \ref{thm:main thm}}
It follows from Theorem \ref{thm: continuous bubbling} that for each $s$ close to $T$, there is $\mu_s > 0$ for which
\begin{equation}
     \lim \limits_{s \to T} \left( ||u(\cdot,s) - u_*(\cdot) -  W_{\mu_s}||_{H^1(B_{\delta}(0))} + \frac{\mu_s}{\sqrt{T-s}} \right)\to 0, \label{eq: single bubble, in time}
\end{equation}
If $s$ is sufficiently close to $T$, then $\mu_s < \delta/(2R_0)$ and there is an open neighborhood $I_s$ of $s$ on which
$$
\sup_{t \in I_s}||u(\cdot,t) - u_*(\cdot) -  W_{\mu_s}||_{H^1(B_{\delta}(0))} + \frac{\mu_s}{\sqrt{T-s}} \leq \frac{\delta_1}{2}
$$
for the constants defined in Proposition \ref{prop: orbital stability}. One obtains from Orbital Stability a remainder $\tilde{g}_s$ (as in Section \ref{sec:orbital stability}), as well as a $C^1$-scale $\lambda_s(t)$, $e^{-1}\mu_s \leq \lambda_s(t) \leq 2 \mu_s$,  defined on $I_s$ such that $u(\cdot,\tau) -  (1-\chi_{r_0})u_*(\cdot) \in B_{\delta_1}(W_{\mu_s})$ for all $\tau \in I_s$,
$$
u(t,r) - (1-\chi_{r_0})u_*(r) = W_{\lambda_s(t)} + g_s, \quad 
\langle \tilde{g}_s(t,\cdot), Z_{\lambda_s(t)}\rangle_{L^2(B_{\delta}(0))} = 0, \quad t \in I_s,
$$
and $\tilde{g}_s = g_s + (1-\chi_{r_0})u_*(r)$ satisfies the inequalities (\ref{eq: smallness of inner product with g_n}) on $I_s$.

Define $\lambda(t) = \lambda_s(t)$ if $t \in I_s$. Then $\lambda(t)$ is well-defined because if $t \in I_s \cap I_{s'}$, then Proposition  \ref{prop: orbital stability} ensures $\lambda_s(t) = \lambda_{s'}(t)$ on that overlap. In particular, $\lambda(s) = \lambda_s(s) \sim \mu_s \ll \sqrt{T-s} \to 0$ as $s \to T$. It follows that $\tilde{g}(t) = \tilde{g}_s(t)$  if $t \in I_s$ is also well-defined. 

Applying the modulation estimates on $\lambda(t)$, this yields a contradiction, as described in Section \ref{sec: applying modulation} (see (\ref{eq: final contradiction})).

\appendix

\section{Profile decompositions in general Hilbert spaces}

The goal of this appendix is to introduce the notion of \textbf{group of dislocations} acting on a Hilbert space and state the Fieseler, Schindler and Tintarev abstract profile decomposition for bounded sequences in $H$. It is the natural generalization of Bahouri and Gérard's profile decomposition. We follow the exposition given by Tao (\cite[Theorem 4.5.3]{tao2013compactness}) and Okumura (\cite[Section 2]{okumura2022profile}).

\begin{definition}[Group of dislocations]
    Let $H$ be a Hilbert space. Let $G$ be a locally sequentially compact group acting unitarily on $H$ and assume that the group action map $G \times H \to H$, $(g,u) \mapsto gu$ is jointly continuous. We say that $G$ is a group of \textit{dislocations} if for any sequence $(g_n) \in G$ which escapes any compact set, the sequence of induced isometries converges operator-weakly to zero, i.e., $g_nu \to 0$ weakly in $H$ for any $u \in H$.
\end{definition}

\begin{remark}
    The trivial group is a group of dislocations on any Hilbert space $H$. The translation action $u(x) \mapsto u(x - x_0)$, $x_0 \in \mathbb R^n$, and the scaling action $u(x) \mapsto \lambda^{-\frac{n-2}{2}}u(x/\lambda)$, $\lambda > 0$, (or the composition of both) create a group of dislocations for the homogeneous Sobolev space $\dot{H}^1(\mathbb R^n)$, as well as $L^{\frac{2n}{n-2}}(\mathbb R^n)$ (this is technically not a Hilbert space, but the definition can be generalized, see Remark \ref{rem:generalization dislocation}).
\end{remark}

\begin{theorem}[Profile Decomposition]\label{thm:abstract profile decomposition}
    Let $H$ be a Hilbert space and $G$ be a group of dislocations acting on $H$. Let $(x_n)_{n \in \mathbb N} \subset H$ be a bounded sequence. After passing to a subsequence (still denoted with index $n$) of $(x_n)_{n \in \mathbb N}$, there exist profiles $(\phi_{m})_{m \in \mathbb N} \subset H$, group elements $((g_{m,n})_{m \in \mathbb N})_{n \geq m} \subset G$, remainders  $((w_{m,n})_{m \in \mathbb N})_{n \geq m} \subset H$ for which
$$
x_{n} = \sum_{m=1}^M g_{m,n} \phi_m + w_{M,n}, \quad M \in \mathbb N_{\geq 1}, n \geq M.
$$
Moreover, one has:
\begin{enumerate}
    \item \textbf{Pythagorean decomposition of energy:}
    $$
    \limsup_{n \to +\infty} ||x_{n}||^2  =  \sum_{m=1}^M||\phi_m||^2 + \limsup_{n \to +\infty} ||w_{M,n}||^2, \quad M \in \mathbb N_{\geq 1}.
    $$
    \item \textbf{Finiteness of decomposition and non-zero profiles:} Let $w_{0,n} = x_n$ and
    \begin{equation}
    D[(x_n)] := \sup\{||\phi|| : \exists \phi \in H, \exists (n_k)_{k \in \mathbb N} \subset \mathbb N, (g_k)_{k \in \mathbb N} \subset G, g^{-1}_{k}x_{n_k} \to \phi \text{ weakly } \}. \label{eq:D(x_n) def}
\end{equation}
    For $M \in \mathbb N_{\geq 0}$, the extracted profile $\phi_{M+1}$ is non-zero if and only if $D[(w_{M,n})] > 0$. Moreover, if $D[(w_{M,n})] = 0$, then all the subsequent profiles $\phi_{M+1} = \phi_{M+2} = ... = 0$ are zero and the decomposition is finite. 
    \item \textbf{Separation of scales:} Fix $i \neq j$ and assume that $\phi_i, \phi_j$ are non-zero profiles. Then $(g^{-1}_{i,n} g_{j,n})_{n \in \mathbb N}$ escapes any compact set.
    \item \textbf{Weak convergence to profile:} For any $M \in \mathbb N_{\geq 1}$, if either $M = 1$ or $\phi_M$ is a non-zero profile, then $g^{-1}_{M,n}x_n \to \phi_M$ weakly.
    \item \textbf{Smallness of remainder:} For $1 \leq m \leq M$, $g^{-1}_{m,n}w_{M,n} \to 0$ weakly. Moreover, the remainders also converge weakly to zero in the following sense:
    $$
    \lim_{M\to +\infty}\sup_{||y|| = 1}  \limsup_{n \to +\infty} \sup_{g \in G} |\langle g^{-1}w_{M,n}, y\rangle| = 0.
    $$
\end{enumerate}
\end{theorem}

\begin{remark}\label{rem:generalization dislocation}
    The theorem generalizes to Banach spaces and to dislocation sets (i.e. without the group structure), see the book of Tintarev (\cite{tintarev2020concentration}), as well as the papers \cite{okumura2022profile} and \cite{schindlertintarev2002abstract}, for an exposition on the topic.
\end{remark}

\begin{remark}
    Using a trivial group of dislocations, one recovers Banach-Alaoglu's theorem on $H$.
\end{remark}

As a corollary, we show that one can always choose a decomposition for which the first extracted profile is the weak limit of the subsequence. 

\begin{corollary}[Profile Decomposition with explicit weak limit]\label{thm:abstract profile decomposition with weak limit first}
    Let $H$ be a Hilbert space and $G$ be a group of dislocations acting on $H$. Let $(x_n)_{n \in \mathbb N} \subset H$ be a bounded sequence. After passing to a subsequence (still denoted with index $n$) of $(x_n)_{n \in \mathbb N}$ converging weakly to $x_* \in H$, there exist profiles $(\phi_{m})_{m \in \mathbb N} \subset H$, group elements $((g_{m,n})_{m \in \mathbb N})_{n \geq m} \subset G$, remainders  $((w_{m,n})_{m \in \mathbb N})_{n \geq m} \subset H$ for which
$$
x_{n} = x_* + \sum_{m=2}^M g_{m,n} \phi_m + w_{M,n}, \quad M \in \mathbb  N_{\geq 1}, n \geq M .
$$
Moreover, one has:
\begin{enumerate}
    \item \textbf{Pythagorean decomposition of energy:}
\begin{align*}
    \limsup_{n \to +\infty} ||x_{n}||^2  &=  ||x_*||^2 + \sum_{m=2}^M||\phi_m||^2 + \limsup_{n \to +\infty} ||w_{M,n}||^2, \quad M \in \mathbb N_{\geq 1}. \\
        \limsup_{n \to +\infty} ||x_{n}-x_*||^2  &= \sum_{m=2}^M||\phi_m||^2 + \limsup_{n \to +\infty} ||w_{M,n}||^2, \quad M \in \mathbb N_{\geq 1}.
\end{align*}
    \item \textbf{Finiteness of decomposition and non-zero profiles:} Let $w_{1,n} = x_n-x_*$ and
    \begin{equation}
    D[(x_n)] := \sup\{||\phi|| : \exists \phi \in H, \exists (n_k)_{k \in \mathbb N} \subset \mathbb N, (g_k)_{k \in \mathbb N} \subset G, g^{-1}_{k}x_{n_k} \to \phi \text{ weakly } \}.
\end{equation}
    For $M \in \mathbb N_{\geq 1}$, the extracted profile $\phi_{M+1}$ is non-zero if and only if $D[(w_{M,n})] > 0$. Moreover, if $D[(w_{M,n})] = 0$, then all the subsequent profiles $\phi_{M+1} = \phi_{M+2} = ... = 0$ are zero and the decomposition is finite. 
    \item \textbf{Separation of scales:} Let $g_{1,n} = \mathrm{Id}_G$. Fix $i \neq j$ and assume either $\phi_i, \phi_j$ are non-zero profiles or $i = 1$ and $\phi_j$ is a non-zero profile. Then, $(g^{-1}_{i,n} g_{j,n})_{n \in \mathbb N}$  escapes any compact set.
    \item \textbf{Weak convergence to profile:} For any $M \in \mathbb N_{\geq 1}$, if either $M = 1$ or $\phi_M$ is a non-zero profile, then $g^{-1}_{M,n}x_n \to \phi_M$ weakly. In particular, $x_n \to \phi_1 = x_*$ weakly.
    \item \textbf{Smallness of remainder:} For $1 \leq m \leq M$, $g^{-1}_{m,n}w_{M,n} \to 0$ weakly. Moreover, the remainders also converge weakly to zero in the following sense:
    $$
    \lim_{M\to +\infty}\sup_{||y|| = 1}  \limsup_{n \to +\infty} \sup_{g \in G} |\langle g^{-1}w_{M,n}, y\rangle| = 0.
    $$
\end{enumerate}
\end{corollary}

\begin{proof}
    Let $x_*$ be a weak limit for $x_n$ after passing to a first subsequence. Apply Theorem \ref{thm:abstract profile decomposition}  using the trivial dislocation group. Hence,
    $$
    x_n = x_* + (x_n - x_*)
    $$
    with energy decoupling. Apply Theorem \ref{thm:abstract profile decomposition} to $x_n -x_*$:
  \begin{equation}
      x_{n} -x_* = \sum_{m=2}^M g_{m,n} \phi_m + w_{M,n}, \quad M \in \mathbb N_{\geq 1}, n \geq M, \label{eq: profile decomposition of x_n-x_*}
  \end{equation}
where we order the profiles starting at $m = 2$ and define $g_{1,n} = \mathrm{Id}_G$.  The energy decoupling follows easily.

Next, we know from the theorem that $(g_{i,n}^{-1}g_{j,n})_{n \in \mathbb N}$ is not contained in any compact set for $i \neq j$, $i,j \geq 2$, if $\phi_i, \phi_j$ are non-zero profiles. Moreover, $g_{j,n}^{-1}(x_n-x_*) \to \phi_j$ weakly if $j = 2$ or $\phi_j$ is a non-zero profile. We first show that $(g_{j,n})_{n \in \mathbb N} = (g_{1,n}^{-1}g_{j,n})_{n \in \mathbb N}$ escapes any compact set for $j \neq 1$ if $\phi_j$ is a non-zero profile. Assume not. After passing to a further subsequence, one can assume that $g_{j,n}$ converges to a group element $g_{j,*} \in G$ in the group topology. As the group action is jointly continuous and $G$ acts unitarily,
$$
\langle g_{j,n}^{-1}(x_n-x_*), \phi \rangle_{H} -\langle g_{j,*}^{-1}(x_n-x_*), \phi \rangle_{H} = \langle (x_n-x_*), (g_{j,n}-g_{j,*})\phi \rangle_{H} \to 0 \quad \forall \phi \in H 
$$
by Cauchy-Schwarz, while 
$$
\langle g_{j,*}^{-1}(x_n-x_*), \phi \rangle_{H} = \langle (x_n-x_*), g_{j,*} \phi \rangle_{H} \to 0 \quad \forall \phi \in H
$$
by weak convergence of $x_n \to x_*$. In other words, $g_{j,n}^{-1}(x_n-x_*) \to 0$ weakly. Yet, by construction of the profiles, $g_{j,n}^{-1}(x_n-x_*) \to \phi_j$ weakly, contradicting the fact that $\phi_j$ is  a non-zero profile.

Furthermore, we show that $g_{j,n}^{-1} x_n \to \phi_j$ weakly if $\phi_j$ is a non-zero profile. We know that $g_{j,n}^{-1} (x_n-x_*) \to \phi_j$ weakly by construction and $g_{j,n}^{-1}$ escapes any compact set. By definition of a dislocation group, $g_{j,n}^{-1}x_* \to 0$ weakly, which proves the claim.

Finally, we need to show that $g_{1,n}^{-1}w_{M,n} = w_{M,n} \to 0$ weakly. Looking at the decomposition (\ref{eq: profile decomposition of x_n-x_*}), it is sufficient to prove that $g_{j,n}\phi_j \to 0$ if $\phi_j$ is a non-zero profile. This follows from the definition of a dislocation group, since $(g_{j,n})_{n \in \mathbb N}$ escapes any compact set. This finishes the proof. 
\end{proof}

 \section*{Declaration}
 \textbf{Data Availability:} Data sharing is not applicable as no datasets were generated or analyzed during the current study.

\textbf{Conflict of interest:} The authors declare that they have no conflict of interest.

\vfill

\nocite{*}
\bibliography{sn-bibliography}

\textbf{Dr. Dylan Samuelian} \\
Ecole Polytechnique Fédérale de Lausanne (EPFL) \\
dylan.samuelian@epfl.ch

\end{document}